\documentclass[12pt,a4paper,reqno%
]{amsart}
\usepackage[T1]{fontenc}
\usepackage{lmodern}
\usepackage[margin=27mm]{geometry}
\usepackage{amsmath,amssymb,amsthm,mathtools}
\usepackage{calrsfs}
\usepackage{enumitem,booktabs,array,longtable}
\usepackage{microtype,needspace}
\usepackage{tikz-cd}
\newcounter{proofstep}
\newcounter{proofanchor}
\newcommand{\proofstep}[1]{\stepcounter{proofanchor}\refstepcounter{proofstep}\label{#1}\textup{(\theproofstep)}}
\usepackage[hidelinks]{hyperref}

\hypersetup{pdftitle={Derived equivalence classification of twisted and generalized multifold extensions of affine type},pdfsubject={Derived and stable equivalence classification of twisted and generalized multifold extensions of affine type}}

\numberwithin{equation}{section}

\setlist[enumerate,1]{label=\textup{(\arabic*)},ref=\arabic*,itemsep=3pt,topsep=5pt}
\setlist[enumerate,2]{label=\textup{(\alph*)},ref=\alph*,itemsep=3pt,topsep=5pt}

\theoremstyle{plain}
\newtheorem{thm}{Theorem}[section]
\newtheorem{prp}[thm]{Proposition}
\newtheorem{lem}[thm]{Lemma}
\newtheorem{cor}[thm]{Corollary}
\theoremstyle{definition}
\newtheorem{dfn}[thm]{Definition}
\newtheorem{exm}[thm]{Example}
\newtheorem{rmk}[thm]{Remark}
\newcommand{\thmref}[1]{Theorem~\ref{#1}}
\newcommand{\prpref}[1]{Proposition~\ref{#1}}
\newcommand{\lemref}[1]{Lemma~\ref{#1}}
\newcommand{\corref}[1]{Corollary~\ref{#1}}
\newcommand{\dfnref}[1]{Definition~\ref{#1}}

\newcommand{\rmkref}[1]{Remark~\ref{#1}}

\newcommand\al{\alpha}
\newcommand\be{\beta}
\newcommand\ga{\gamma}

\newcommand\et{\eta}
\renewcommand\th{\theta}
\newcommand\la{\lambda}
\newcommand\ro{\rho}
\newcommand\si{\sigma}
\newcommand\ta{\tau}

\newcommand\ph{\phi}
\newcommand\ps{\psi}
\newcommand\Ga{\Gamma}
\newcommand\De{\Delta}
\newcommand\La{\Lambda}

\newcommand\Om{\Omega}
\newcommand\Ph{\Phi}
\newcommand\Ps{\Psi}

\newcommand\Hom{\operatorname{Hom}}
\newcommand\Aut{\operatorname{Aut}}
\newcommand\End{\operatorname{End}}
\newcommand\Ext{\operatorname{Ext}}
\newcommand\Ker{\operatorname{Ker}}

\renewcommand\mod{\operatorname{mod}}
\newcommand{\umod}{\operatorname{\underline{mod}}\nolimits}
\newcommand{\uHom}{\operatorname{\underline{Hom}}\nolimits}
\newcommand\prj{\operatorname{prj}}
\newcommand\rank{\operatorname{rank}}
\newcommand\Kb{{\mathcal K}^{\text{\rm b}}}
\newcommand\Db{{\mathcal D}^{\text{\rm b}}}

\newcommand\calA{{\mathcal A}}
\newcommand\calB{{\mathcal B}}
\newcommand\calC{{\mathcal C}}
\newcommand\calD{{\mathcal D}}
\newcommand\calE{{\mathcal E}}

\newcommand\calI{{\mathcal I}}

\newcommand\calM{{\mathcal M}}

\newcommand\calP{{\mathcal P}}

\newcommand\calS{{\mathcal S}}
\newcommand\calT{{\mathcal T}}

\newcommand\bbZ{{\mathbb Z}}

\newcommand\bbC{{\mathbb C}}
\renewcommand\k{\Bbbk}
\newcommand\ds{\oplus}
\newcommand\ox{\otimes}
\newcommand\Ds{\bigoplus}
\newcommand\ovl{\overline}

\newcommand\Out{\operatorname{Out}}
\newcommand\Inn{\operatorname{Inn}}
\newcommand\Pic{\operatorname{Pic}}
\newcommand\DPic{\operatorname{DPic}}
\newcommand\typ{\operatorname{typ}_{\mathrm{aff}}}
\newcommand\kk{\k^\times}
\newcommand{\Aaff}[2]{\widetilde A_{#1,#2}}
\DeclareMathOperator{\PGL}{PGL}
\DeclareMathOperator{\GL}{GL}
\DeclareMathOperator{\Aff}{Aff}

\newcommand\blank{\text{-}}
\newcommand\dereq{\overset{\text{der}}{\sim}}

\newcommand\Mod{\operatorname{Mod}}
\newcommand\Cdot{\raisebox{1pt}{\scalebox{0.4}{$\bullet$}}}
\newcommand\down{_{\Cdot}}
\newcommand\up{^{\Cdot}}
\newcommand{\To}{\Rightarrow}
\newcommand\id{1\kern-.25em{\text{{\rm l}}}}

\title[Multifold extensions of affine type]{Derived equivalence classification of twisted and generalized
multifold extensions of affine type}
\author{Hideto Asashiba}
\address{Department of Mathematics, Faculty of Science,
Shizuoka University, 836 Ohya, Suruga-ku, Shizuoka, 422-8529, Japan}
\address{Institute for Advanced Study, KUIAS, Kyoto University,
Yoshida Ushinomiya-cho, Sakyo-ku, Kyoto 606-8501, Japan}
\address{Osaka Central Advanced Mathematical Institute,
3-3-138 Sugimoto, Sumiyoshi-ku, Osaka, 558-8585, Japan}
\email{asashiba.hideto@shizuoka.ac.jp}
\date{}

\begin{document}
\begin{abstract}
Let $B$ be a piecewise hereditary algebra of affine type over an
algebraically closed field. In this paper we give a derived and stable
equivalence classification of twisted multifold extensions of $B$.
The types $\widetilde D$ and $\widetilde E$ are already covered by the
classification for tree type. We therefore consider the remaining types
$\Aaff pq$. We show that the $n$th power of the Nakayama functor on the
full subcategory of a transjective component determines the twist,
including its action on morphisms, up to conjugacy in the derived Picard
group of a hereditary algebra. Together with the hereditary derived type
and the folding number, this gives a complete invariant for derived
equivalence, and also for stable triangle equivalence. We then compute
these conjugacy classes in the unbalanced, balanced, bypass, and Kronecker
cases, keeping the finite rotation data. Finally, by the normalization
theorem of Asashiba--Kimura--Nakashima--Yoshiwaki, we obtain the same
classification for generalized multifold extensions with nonzero integer
jump.
\end{abstract}

\maketitle

\section{Introduction and the classification theorem}

The derived and stable equivalence classification of twisted multifold
extensions of piecewise hereditary algebras of tree type was established
in \cite{Asa}. Its type consists of the underlying tree, the folding
number, and a conjugacy class of graph automorphisms. The affine diagrams
$\widetilde D_r$ and $\widetilde E_6,\widetilde E_7,\widetilde E_8$ therefore
require no new tree-type argument. The remaining affine case is an
acyclic orientation of a cycle. In this case the action on the vertices
of the Auslander--Reiten quiver does not, in general, determine the twist:
an additional parameter can occur in the action on morphisms.

We first recall the notation for multifold extensions.
{Let $B$ be a finite-dimensional basic algebra over an algebraically closed field $\k$, and let}
$\widehat B$ be its repetitive category.
For each $i\in\bbZ$, the objects $x^{[i]}=(x,i)$,
with $x$ an object of $B$, form a full subcategory
$B^{[i]}$ isomorphic to $B$.
We call $B^{[i]}$ the {\emph{$i$th layer}} of $\widehat B$.
The Nakayama automorphism $\nu_B$ sends $x^{[i]}$
to $x^{[i+1]}$ and hence sends $B^{[i]}$ to $B^{[i+1]}$.

For $\ph\in\Aut(B)$, let $\widehat\ph$ be its standard
lift to $\widehat B$, which preserves the layer index.
{For each integer $n\geq1$, the \emph{twisted $n$-fold extension}} of $B$ by
$\ph$ is
\[
 T_\ph^n(B)
 :=\widehat B/\langle\widehat\ph\nu_B^n\rangle.
\]
\begingroup
We call $B$ the \emph{base algebra} and $n$ the
\emph{folding number} of this presentation of $T_\ph^n(B)$.
\par\endgroup
Here the quotient is the classical orbit category,
which we identify with its associated basic algebra.

More generally, we say that an automorphism $g$ of
$\widehat B$ has \emph{jump} $n\in\bbZ$ if
\[
 g(B^{[i]})=B^{[i+n]}
 \qquad\text{for every }i\in\bbZ.
\]
Thus $g$ shifts every layer by the same integer $n$.
For such an automorphism with $n\ne0$, we call
$\widehat B/\langle g\rangle$ a
\emph{generalized multifold extension} of $B$.
In particular, $\widehat\ph\nu_B^n$ has jump $n$,
so twisted multifold extensions are special cases.
Replacing $g$ by $g^{-1}$ does not change the
orbit algebra and allows us to assume that $n>0$.
We give the detailed constructions and conventions
in Section~2.

We distinguish our terminology from that of \cite{BHS04}.
In that paper, a selfinjective algebra of \emph{Euclidean type}
is an orbit algebra $\widehat B/G$ with $B$ tilted from a
hereditary algebra whose underlying graph is an extended Dynkin
diagram and $G$ an admissible group of automorphisms.
In the present paper, the base algebra is only required to be
\emph{piecewise hereditary}: by this we mean that it is derived
equivalent to a hereditary algebra.
When that hereditary algebra has an extended Dynkin diagram, we
say that the base algebra is of \emph{affine type}.
The extensions considered here inherit this designation from
their base algebras.
Thus we use \emph{Euclidean type} when discussing the tilted-base
setting of \cite{BHS04}, and \emph{affine type} for our
piecewise hereditary setting.
This is a convention of the present paper, not a distinction
between two kinds of underlying diagrams.
We do not require the base algebra to be tilted.
Our arguments use a derived equivalence with a hereditary algebra
together with Happel's equivalence for the repetitive algebra.

The motivation for the present paper comes from the counterexamples of
Asashiba, Itaba and Saor\'in \cite{HH}.
Bocian, Holm and Skowro\'nski \cite[Theorem~1]{BHS04}
claimed a complete derived equivalence classification of weakly
symmetric algebras of Euclidean type with singular Cartan matrix.
The family $(\La_\la)_{\la\in\k^\times}$ introduced in \cite{HH}
satisfies these hypotheses.
However, its members with $\la\ne1$ are missing from that
classification up to derived equivalence.
By \cite[Propositions~2.1 and~2.4]{HH}, the algebras $\La_\la=T^1_{\ph_\la}(B)$,
where $B$ is hereditary of type $\Aaff32$, have isomorphic stable
Auslander--Reiten quiver for all $\la\in \k^\times$. On the other hand, \cite[Theorem~3.1]{HH} gives
\[
 \dim_\k HH^2(\La_\la)=
 \begin{cases}
 1,&\la\neq1,\\
 2,&\la=1.
 \end{cases}
\]
Hence $\La_\la$ with $\la\neq1$ is not derived equivalent
to $\La_1$.
This non-equivalence also follows from
\cite[Proposition~2.2]{HH}, which states that $\La_\la$
is symmetric if and only if $\la=1$, because symmetry
is preserved under derived equivalence.
These counterexamples show that the stable
Auslander--Reiten quiver does not contain all the information required
for the classification.

In this paper we record the action on morphisms and classify,
up to derived equivalence, twisted multifold extensions together
with generalized multifold extensions having nonzero integer jump,
over piecewise hereditary bases of affine type.
This family includes all the algebras $\La_\la$ above.
We do not restrict the base to be tilted, nor do we impose weak
symmetry on the extensions.
In particular, we prove that, {for all $\la,\mu\in\k^\times$,}
\[
 \La_\la\dereq\La_\mu
 \quad\Longleftrightarrow\quad \la=\mu
\]
(see \corref{cor:HH-family}).
This proves the conjecture stated in the acknowledgments of
\cite{HH}.
A generalization of the results of this paper will be treated
in a subsequent paper.

We use the term \emph{transjective component} as follows.
For a connected hereditary algebra $H$ of infinite
representation type, let $\calT$ be the Auslander--Reiten
component of $\Db(\mod H)$ containing the indecomposable
projective modules. Its vertices are represented by the
indecomposable preprojective modules and the objects $I[-1]$
with $I$ indecomposable preinjective.
We call $\calT$ and its shifts the transjective components.
These are the irregular components in the terminology of
\cite[Definition~2.2]{MY}.
For the selfinjective extensions considered below, we use
the same term for the corresponding stable Auslander--Reiten
components obtained through Happel's equivalence and the
stable pushdown functor.

We first identify the full subcategory on a
transjective component of the stable module category of a twisted
$n$-fold extension with the mesh category of the repetition of a
hereditary quiver. We then show that the $n$th power of the Nakayama
functor determines the inverse twist on this category. By the
restriction theorem of Miyachi--Yekutieli \cite{MY}, equivalence of
these functors is expressed as conjugacy in the derived Picard group.
The converse follows from equivariant tilting.

Throughout the main text we fix an algebraically closed field $\k$.
All algebras are assumed to be finite-dimensional, basic, and connected,
and all functors and equivalences to be $\k$-linear.

{A parenthesized range attached to a formula means that the formula holds for all values in that
range. For a defining assignment, it means that the assignment is made for all such values. Existential
quantifiers are always stated explicitly.}

By a
\emph{stable equivalence} we mean a triangle equivalence between the
stable module categories of selfinjective algebras.
(Appendix~\ref{app:additive-stable} shows that the triangle
assumption can be omitted in the classification theorem.
We give this refinement separately so that readers interested
only in the derived equivalence classification may omit that appendix.)

We use the following set of hereditary derived types:
\begin{equation}\label{eq:types}
 \De\in
\{\Aaff pq\mid {p,q\in\bbZ,\ p\geq q\geq1}\}
 \ \cup\ \{\widetilde D_r\mid {r\in\bbZ,\ r\geq4}\}
 \ \cup\ \{\widetilde E_6,\widetilde E_7,\widetilde E_8\}.
\end{equation}
Here $\Aaff pq$ has $p+q$ vertices and {records the numbers of arrows
in the two cyclic directions;} $\Aaff11$ denotes the Kronecker type. Thus
$\De$ records more than the unoriented affine diagram in type
$\widetilde A$.

A triangle equivalence
$F:\Db(\mod B)\longrightarrow\Db(\mod C)$
is called \emph{standard} if it is naturally isomorphic, as a
triangle functor, to $-\ox_B^{\mathbf L}X$ for some two-sided
tilting complex $X$ of $B$-$C$-bimodules; see \cite{Ric91}.
For each equivalence $F$, we choose a quasi-inverse and denote
it by $F^-$, so that
\[
 F^-F\mathrel{{\cong}}\mathrm{Id},\qquad FF^-\mathrel{{\cong}}\mathrm{Id}
\]
naturally as triangle functors.

We regard $\DPic_\k(H)$ as the group of natural-isomorphism
classes of standard derived autoequivalences of
$\Db(\mod H)$, with multiplication induced by composition.
For such an autoequivalence $G$, we denote its
natural-isomorphism class by $\operatorname{cl}(G)$.

{Fix a hereditary algebra $H_\De$ for each type $\De$ in \eqref{eq:types}.}
If $B$ is piecewise hereditary of type $\De$, choose a
standard equivalence
\[
 F:\Db(\mod B)\longrightarrow\Db(\mod H_\De).
\]
For $\ph\in\Aut(B)$, let $\Ph$ be the induced derived
autoequivalence, with the convention that
$\Ph B(\blank, x)\mathrel{{\cong}} B(\blank, \ph x)$ {for every object $x$ of $B$.} Set
\begin{equation}\label{eq:gamma}
 \ga_{B,\ph}=\operatorname{cl}(F\Ph F^-)
 \in\DPic_\k(H_\De).
\end{equation}
Since any two quasi-inverses of $F$ are naturally isomorphic,
this element is independent of the choice of $F^-$.

Let $F'$ be another standard equivalence with the same source
and target. Put $U=F'F^-$ and choose
$U^-=F(F')^-$ as a quasi-inverse of $U$. Then
\[
 F'\Ph(F')^-\mathrel{{\cong}} U(F\Ph F^-)U^-
\]
naturally as triangle functors. Hence the element defined
using $F'$ is
\[
 \operatorname{cl}(U)\,\ga_{B,\ph}\,
 \operatorname{cl}(U)^{-1}.
\]
Thus the conjugacy class of $\ga_{B,\ph}$ is independent
of the choice of $F$.

From now on, when an autoequivalence $F$ is regarded as an
element of $\DPic_\k(H)$, we write $F$ for
$\operatorname{cl}(F)$ and $F^{-1}$ for
$\operatorname{cl}(F)^{-1}=\operatorname{cl}(F^-)$.
When an actual quasi-inverse functor is needed, we retain
the notation $F^-$.
For an element $\ga\in\DPic_\k(H)$, the brackets $[\ga]$
denote its conjugacy class. In particular, $[F]$ denotes
the conjugacy class of $\operatorname{cl}(F)$ under this
convention.

\begin{dfn}\label{def:type}
For a presentation $\La=T_\ph^n(B)$, $n\geq1$, define
\[
 \typ(\La)=(\De,n,[\ga_{B,\ph}]).
\]
Here the subscript $\mathrm{aff}$ stands for affine.
The independence of this triple from the presentation follows from
\thmref{thm:main}. 
The third component is not an arbitrary conjugacy class in
$\DPic_\k(H_\De)$. It must be represented by $F\Ph F^-$,
where $B$ is a piecewise hereditary algebra of type $\De$,
$\ph\in\Aut(B)$, $\Ph$ is the derived autoequivalence induced
by $\ph$, and $F:\Db(\mod B)\to\Db(\mod H_\De)$ is a
standard equivalence.
\end{dfn}

\begin{thm}\label{thm:main}
Let $\La=T_\ph^n(B)$ and $\La'=T_\ps^{n'}(B')$ be twisted
multifold extensions of piecewise hereditary algebras of affine type.
The following are equivalent:
\begin{enumerate}
\item\label{item:main-derived} $\La$ and $\La'$ are derived equivalent;
\item\label{item:main-stable} $\La$ and $\La'$ are stably equivalent;
\item\label{item:main-type} $\typ(\La)=\typ(\La')$.
\end{enumerate}
The same criterion holds for generalized multifold extensions with
nonzero integer jump, after choosing the generator with positive jump
and applying the normalization in \thmref{thm:normalization}.
\end{thm}

In tree type the third entry is exactly the conjugacy-class datum in
\cite[Definition~8.2]{Asa}. In type $\widetilde A$ the explicit
possibilities are given by \thmref{thm:table}.

The third component of our invariant is a conjugacy class,
not merely the order of a representative.
Even for graph automorphisms, two elements of the same order
need not be conjugate. For example, the automorphism group
of the graph $\widetilde D_4$ is $S_4$, acting on its four
outer vertices. The permutations $(12)$ and $(12)(34)$
both have order two, but they are not conjugate.
Thus replacing the conjugacy class by the order would lose
information needed for the classification.

A further point occurs in the balanced case $p=q=m>1$.
Consider the twists whose graph part is a rotation.
Write their parameters as $(\la,a)$, where $\la\in\kk$
is the scalar parameter and $a\in{\bbZ/m\bbZ}$ specifies the
power $r^a$ of a rotation $r$ of order $m$.
As shown in \thmref{thm:table}, two such pairs represent
the same conjugacy class precisely when they are equal
or are related by
\[
 (\la,a)\longmapsto(\la^{-1},-a).
\]
The scalar and the rotation are therefore inverted
simultaneously, not independently.
For example, over $\bbC$ with $m=3$, the pair $(2,1)$
is identified with $(1/2,2)$, but with neither $(2,2)$
nor $(1/2,1)$.
Consequently, recording the scalar up to inversion and
the rotation up to inversion separately would identify
distinct conjugacy classes.

The new part of the proof is the $\widetilde A$ case. Combining this
with the tree-type classification and the normalization of generalized
extensions gives the statement for all affine types. Within the
family in \thmref{thm:main}, we impose no further assumption of
symmetry, polynomial growth, or a specified representation type.
The classification concerns this family of extensions.

This paper is organized as follows. In Section~2 we recall the
conventions and results used in the proof. Section~\ref{sec:tree}
explains the tree cases. In Sections~\ref{sec:mesh} and~\ref{sec:reduction}
we prove the criterion in type $\widetilde A$ and the
reduction to hereditary algebras.
We compute the parameters in Section~\ref{sec:parameters}
and list the tree-type parameters in Section~\ref{sec:treetable}.
Section~\ref{sec:generalized} treats generalized extensions, and
Section~\ref{sec:examples} gives consequences and examples.
Appendix~\ref{app:descent} gives the equivariant tilting argument.
Appendix~\ref{app:standard-facts} recalls the identification of
the Picard group of a basic algebra with its outer automorphism
group and with natural-isomorphism classes of automorphisms
of the associated category.
Appendix~\ref{app:additive-stable} proves that the classification
remains valid for $\k$-linear equivalences of stable module
categories that are not assumed to be triangle equivalences.

\section{Preliminaries and conventions}
\label{sec:conventions}

{A subcategory is \emph{closed under isomorphisms} if every object
of its ambient category isomorphic to one of its objects also belongs to it.
Fullness alone does not impose this condition.}

Let $B$ be a finite-dimensional basic connected $\k$-algebra.
Fix a complete set $B_0$ of pairwise orthogonal primitive
idempotents of $B$. We regard $B$ as the finite $\k$-category
with object set $B_0$ and morphism spaces
$B(x,y)=yBx$ {for all $x,y\in B_0$.}
Throughout this paper, automorphisms of $B$ and of its
repetitive category $\widehat B$ mean automorphisms of
$\k$-categories. They are allowed to permute the objects.

We use the following notation for the
repetitive category $\widehat B$, as in \cite[Definition~1.1]{AK} and
\cite[Definition~1.2]{AKNY}. {We write $D=\Hom_\k(\blank,\k)$ for the vector-space dual.}
Its objects are
\[
 x^{[i]}:=(x,i)\qquad(x\in B_0,\ i\in\bbZ ),
\]
{For all $x,y\in B_0$ and $i,j\in\bbZ$, its morphism spaces are given by}
\[
 \widehat B(x^{[i]},y^{[j]})=
 \begin{cases}
 B(x,y),&j=i,\\
 DB(y,x),&j=i+1,\\
 0,&\text{otherwise}.
 \end{cases}
\]
\begingroup
{For each $i\in\bbZ$, we call $B^{[i]}$ and $B^{[i+1]}$ \emph{adjacent layers}.}
A morphism $x^{[i]}\to y^{[i]}$ is called a
\emph{same-layer morphism}, and a morphism
$x^{[i]}\to y^{[i+1]}$ an \emph{adjacent-layer morphism}.
\par\endgroup
{For all $x,y\in B_0$, $i\in\bbZ$, and $f\in B(x,y)$ (resp.\ $\be\in DB(y,x)$), we denote the}
corresponding morphism from $x^{[i]}$ to $y^{[i]}$
(resp.\ from $x^{[i]}$ to $y^{[i+1]}$) by $f^{[i]}$
(resp.\ $\be^{[i]}$). 
{
{We write $gf=g\circ f$ for composition.} To specify composition in
$\widehat B$, consider composable morphisms
\[
 x^{[i]}\xrightarrow{f}y^{[j]}\xrightarrow{g}z^{[\ell]}.
\]
There are four cases in which both morphism spaces can be nonzero.
In the formulas below, the same letters denote the corresponding
morphisms or linear functionals in $B$ and $DB$.
\begin{enumerate}
\item\label{item:composition-same} If $j=\ell=i$, then $f\in B(x,y)$ and $g\in B(y,z)$.
Their composite is the ordinary composite $gf\in B(x,z)$.
Thus $g^{[i]}f^{[i]}=(gf)^{[i]}$.
\item\label{item:composition-left} If $j=i$ and $\ell=i+1$, then
$f\in B(x,y)$ and $g\in DB(z,y)$.
Their composite is the functional $gf\in DB(z,x)$ given by
\[
 (gf)(h)=g(fh)\qquad\bigl(h\in B(z,x)\bigr).
\]
Here $fh\in B(z,y)$, so the evaluation on the right is defined.
\item\label{item:composition-right} If $j=\ell=i+1$, then
$f\in DB(y,x)$ and $g\in B(y,z)$.
Their composite is the functional $gf\in DB(z,x)$ given by
\[
 (gf)(h)=f(hg)\qquad\bigl(h\in B(z,x)\bigr).
\]
Here $hg\in B(y,x)$.
\item\label{item:composition-zero} If $j=i+1$ and $\ell=i+2$, then
$f\in DB(y,x)$ and $g\in DB(z,y)$.
In this case $gf=0$, since
$\widehat B(x^{[i]},z^{[i+2]})=0$.
\end{enumerate}
The second and third cases use the $B$-$B$-bimodule structure of $DB$.

{For each $i\in\bbZ$, we denote by $B^{[i]}$ the full subcategory on the objects $x^{[i]}$ for all
$x\in B_0$.} The embedding of the $i$th layer is
\[
 1^{[i]}:B\longrightarrow B^{[i]}\subseteq\widehat B,
 \qquad x\longmapsto x^{[i]},\quad f\longmapsto f^{[i]}.
\]
{The \emph{Nakayama automorphism} $\nu_B$ is defined, for all $x,y\in B_0$, $i\in\bbZ$, $f\in B(x,y)$,
and $\be\in DB(y,x)$, by}
\[
 \nu_B(x^{[i]})=x^{[i+1]},\qquad
 \nu_B(f^{[i]})=f^{[i+1]},\qquad
 \nu_B(\be^{[i]})=\be^{[i+1]}.
\]
Let $\ph\in\Aut(B)$ be a $\k$-category automorphism.
Its \emph{standard lift} $\widehat\ph\in\Aut(\widehat B)$ is defined {on objects and same-layer
morphisms, for all $x,y\in B_0$, $i\in\bbZ$, and $f\in B(x,y)$, by}
\[
 \widehat\ph(x^{[i]})=\ph(x)^{[i]},\qquad
 \widehat\ph(f^{[i]})=\ph(f)^{[i]}.
\]
{For all $x,y\in B_0$ and $\be\in DB(y,x)$, define}
$\be_\ph\in DB(\ph(y),\ph(x))$ by
\[
 \be_\ph(h)=\be\bigl(\ph^{-1}(h)\bigr)
 \qquad\bigl(h\in B(\ph(y),\ph(x))\bigr){.}
\]
{The standard lift on adjacent-layer morphisms is given by}
\[
 \widehat\ph(\be^{[i]})=(\be_\ph)^{[i]}{\quad(i\in\bbZ).}
\]
Thus the action on these functionals is the dual of $\ph^{-1}$.
The standard lift preserves every layer index and commutes with
$\nu_B$. In particular, the following square commutes strictly:
\[
\begin{tikzcd}[column sep=large]
 B^{[i]} \arrow[r,"\widehat\ph"] \arrow[d,"\nu_B"'] &
 B^{[i]} \arrow[d,"\nu_B"] \\
 B^{[i+1]} \arrow[r,"\widehat\ph"'] & B^{[i+1]}.
\end{tikzcd}
\]
{The notation $X[i]$ denotes the shift of a complex.
It is distinct from the notation $x^{[i]}$ for an object
in the $i$th layer.}

We use the classical {\emph{$\k$-linear orbit category}}.
Let $g$ be an automorphism of $\widehat B$ with nonzero
integer jump, and set $G=\langle g\rangle$.
The objects of $\widehat B/G$ are the $G$-orbits
of objects of $\widehat B$.
For each orbit, choose one representative object.
If $X$ is the chosen representative, write
\[
 GX=\{g^rX\mid r\in\bbZ\}
\]
for its orbit.
{For any chosen representatives $X,Y$, we identify}
\[
 (\widehat B/G)(GX,GY)
 \mathrel{{\cong}}\Ds_{r\in\bbZ}\widehat B(g^rX,Y).
\]
{For any chosen representatives $X,Y,Z$, any $r,s\in\bbZ$, and morphisms $u_r:g^rX\to Y$ and
$v_s:g^sY\to Z$, their composite is represented by}
\[
 v_s\ast u_r=v_sg^s(u_r):g^{s+r}X\longrightarrow Z.
\]
Composition on arbitrary morphisms is obtained by
bilinearity.
Choosing different representatives gives an isomorphic
$\k$-category.
Its associated basic algebra is the direct sum of its morphism
spaces, with multiplication given by composition and zero for
noncomposable morphisms. 

\begin{dfn}
For $\ph\in\Aut(B)$ and an integer $n\geq1$, the \emph{twisted
$n$-fold extension} is the orbit category, or its associated basic
algebra,
\begin{equation}\label{eq:twisted}
 T_\ph^n(B)=\widehat B/\langle\widehat\ph\nu_B^n\rangle.
\end{equation}
A \emph{generalized multifold extension} considered here is
$\widehat B/\langle g\rangle$, where
$g\in\Aut(\widehat B)$ has a \emph{nonzero integer jump}:
there is an integer $n\ne0$ such that
\[
 g(B^{[i]})=B^{[i+n]}\qquad(i\in\bbZ).
\]
Replacing $g$ by $g^{-1}$ does not change the acting group
and permits us to assume $n>0$.
\end{dfn}

These orbit algebras are finite-dimensional and selfinjective.
{For any small $\k$-linear category $A$, we write $\Mod A$ for the category of all right $A$-modules
and $\mod A$ for its full subcategory of finite-dimensional modules.}
{Here $\Db(\mod B)$ denotes the bounded derived category of $\mod B$.}
We write $\umod\widehat B$ for the stable category of finite-dimensional
right $\widehat B$-modules: its morphisms are module homomorphisms
modulo those that factor through projective modules.
When $B$ has finite global dimension, Happel's equivalence is a
triangle equivalence
\begin{equation}\label{eq:Happel}
 E_B:\Db(\mod B)\xrightarrow{\ \sim\ }\umod\widehat B
 \qquad\text{\cite[Theorem~II.4.9]{Hap}}.
\end{equation}

Let $\Om$ denote the syzygy functor on the stable category.
The translation of this triangulated category is $\Om^{-1}$, so
\[
 E_B\circ[1]\cong\Om^{-1}\circ E_B
\]
as triangle functors. These functors are naturally isomorphic;
we do not identify them as equal functors.

Let $N$ denote the Nakayama functor induced on the stable category,
and let $\ta$ denote the Auslander--Reiten translation.
The natural isomorphism $N\cong\Om^{-2}\ta$ gives, for a hereditary
algebra $H$, the following compatibility with Happel's equivalence:
\[
\begin{tikzcd}[column sep=large,row sep=large]
 \Db(\mod H) \arrow[r,"E_H"] \arrow[d,"{\ta[2]}"'] &
 \umod\widehat H \arrow[d,"N_{\widehat H}"] \\
 \Db(\mod H) \arrow[r,"E_H"'] & \umod\widehat H.
\end{tikzcd}
\qquad
 E_H\circ\ta[2]\cong N_{\widehat H}\circ E_H.
\]
This square commutes up to natural isomorphism. The symbol $\ta$
on the derived side denotes its Auslander--Reiten translation.

We next recall two results used below. Write $\prj B$ for the
category of finitely generated projective right $B$-modules and
$\Kb(\prj B)$ for its bounded homotopy category.
{For a subcategory $\calE$ of a triangulated category,
$\operatorname{thick}(\calE)$ denotes the smallest full triangulated
subcategory of the ambient category that contains $\calE$ and is closed
under isomorphisms in that category and under direct summands.}
A {\emph{tilting spectroid}} for $B$ is a full subcategory $\calE$ of
$\Kb(\prj B)$ on pairwise nonisomorphic indecomposable objects such that
\[
\begin{gathered}
 \Kb(\prj B)(U,V[r])=0
 \quad(U,V\in\calE,\ {r\in\bbZ\setminus\{0\}}),\\
 \operatorname{thick}(\calE)=\Kb(\prj B).
\end{gathered}
\]
For $\ph\in\Aut(B)$, let $\Ph$ be the induced derived
autoequivalence, with the convention
$\Ph B(\blank,x)\cong B(\blank,\ph(x))$ {for every object $x$ of $B$,} used in Section~1.

For finite-dimensional $\k$-algebras $A$ and $B$, we say {that} 
$A$ and $B$ are {\emph{derived equivalent}}, and write
$A\dereq B$ if there is a $\k$-linear triangle equivalence
\[
 \Db(\mod A)\simeq\Db(\mod B).
\]

\begin{prp}[Equivariant tilting]\label{prop:tilting}
Suppose that a tilting spectroid $\calE$ for $B$ is stable under
$\Ph$ up to isomorphisms, and that $\calE\simeq C$.
After replacing $\calE$ by an isomorphic tilting spectroid
that is strictly stable under the induced action,
transport that action to an automorphism $\ps$ of $C$.
Then
\[
 T_\ph^n(B) \dereq T_\ps^n(C)
 \qquad(n\geq1).
\]
{If $C(x,x)=\k$ for every object $x$ of $C$,} replacing the transported action by a
naturally isomorphic action does not change the orbit algebra up to
isomorphism.
\end{prp}

\begin{proof}
\emph{The first assertion.}
We first construct a transported automorphism $\ps$ of
$C$ such that
\[
 T_\ph^n(B)\dereq T_\ps^n(C)
 \qquad(n\geq1).
\]
{Use the model of $\prj B$ whose objects are finite formal direct sums of representable projectives
and whose morphisms are matrices of morphisms. We call it the \emph{formal additive hull}. Work in its bounded
homotopy category, and denote by $\widetilde\ph$ the}
 automorphism induced by $\ph$.
By \cite[Remark~3.5]{Asa}, we may replace $\calE$ by
an isomorphic tilting spectroid $\calE'$ that is strictly
stable under $\widetilde\ph$.
Choose an isomorphism $J:\calE'\to C$ and define
\[
 \ps=J(\widetilde\ph|_{\calE'})J^{-1}.
\]
Here $J$ is an isomorphism of categories, so $J^{-1}$
is its inverse functor.
Then the following diagram commutes:
\[
\begin{tikzcd}[column sep=large]
 \calE' \arrow[r,"J"]
         \arrow[d,"\widetilde\ph|_{\calE'}"'] &
 C \arrow[d,"\ps"] \\
 \calE' \arrow[r,"J"'] & C.
\end{tikzcd}
\]
Therefore \cite[Proposition~3.6]{Asa} gives
$T_\ph^n(B)\dereq T_\ps^n(C)$ for every $n\geq1$.

\medskip
\emph{The last assertion.}
Assume that $C(x,x)=\k 1_x$ for every object $x$ {of $C$.}
Let $\ps'$ be an automorphism of $C$ naturally
isomorphic to $\ps$. We prove that
\[
 T_\ps^n(C)\cong T_{\ps'}^n(C)
 \qquad(n\geq1).
\]
Choose a natural isomorphism
$\alpha:\ps\Rightarrow\ps'$.
{Since $C$ is skeletal, $\ps(x)=\ps'(x)$ for every object $x$ of $C$. For each such $x$, there is a
scalar $c_x\in\k^\times$ such that $\alpha_x=c_x1_{\ps(x)}$.}
{For all objects $x,y$ of $C$ and every $f\in C(x,y)$, naturality gives}
\[
 c_y\ps(f)=\ps'(f)c_x.
\]
Fix $n\geq1$ and put
\[
 g=\widehat\ps\nu_C^n,
 \qquad h=\widehat{\ps'}\nu_C^n.
\]
These automorphisms have the same action on objects and
have jump $n$. In the notation of \cite{AKNY}, their
induced automorphisms $g_0,h_0$ of $C$ are $\ps,\ps'$,
respectively.
Thus the map $x\mapsto c_x$ satisfies the hypothesis
of \cite[Theorem~4.4]{AKNY}.
That theorem gives
\[
 T_\ps^n(C)
 =\widehat C/\langle g\rangle
 \cong\widehat C/\langle h\rangle
 =T_{\ps'}^n(C).
\]
This proves the last assertion.
\end{proof}

For a quiver $Q$, write $\bbZ Q$ for its repetition quiver and
$I_{\mathrm{mesh}}$ for the ideal generated by its mesh relations.
Thus $\k(\bbZ Q,I_{\mathrm{mesh}})$ is the mesh category.
For a $\k$-category $\calB$, $\Out_\k(\calB)$ denotes the group of
autoequivalences modulo natural isomorphism. As in Section~1,
$\DPic_\k(H)$ is the derived Picard group: the group of
natural-isomorphism classes of standard derived autoequivalences.

\begin{prp}[Full mesh category and restriction]\label{prop:MY}
Let $H=\k Q$ be connected hereditary of infinite representation type.
Let $\calB$ be the full subcategory of $\Db(\mod H)$ on representatives
of the transjective component containing the projective modules.
Then
\begin{equation}\label{eq:mesh}
 \calM_Q:=\k(\bbZ Q,I_{\mathrm{mesh}})\simeq\calB,
 \qquad
 \DPic_\k(H)\mathrel{{\cong}}\Out_\k(\calB)\times\bbZ.
\end{equation}
We describe the second isomorphism explicitly.
Let $F$ be a standard derived autoequivalence of $H$.
There is a unique integer $m$ such that $F$ sends the
chosen transjective component to its shift by $m$.
Set
\[
 G=[-m]\circ F.
\]
Then $G$ preserves the chosen component.
For each object $X$ of $\calB$, choose an object
$G_0(X)$ of $\calB$ such that $G(X)\mathrel{{\cong}} G_0(X)$.
Fix an isomorphism
\[
 \theta_X:G(X)\longrightarrow G_0(X).
\]
For each morphism $f:X\to Y$ in $\calB$, define
\[
 G_0(f)=\theta_Y\circ G(f)\circ\theta_X^{-1}.
\]
This defines an autoequivalence
$G_0:\calB\longrightarrow\calB$.
Its natural-isomorphism class does not depend on the
choices of representatives and isomorphisms.
The second isomorphism sends the class of $F$ to
\[
 \bigl(\operatorname{cl}(G_0),m\bigr)
 \in\Out_\k(\calB)\times\bbZ.
\]
In particular, the class of the shift functor $[m]$
corresponds to $(1,m)$. The second factor is central
in $\DPic_\k(H)$.
\end{prp}

\begin{proof}
Apply \cite[Theorems~2.6 and~3.7]{MY}.
The first assertion is an equivalence with the full subcategory
$\calB$, including all its morphisms, not only irreducible morphisms.
\end{proof}

\begingroup
\begin{rmk}[Outer automorphism groups]\label{rmk:outer-groups}
{
$\Out(B)$ is the outer automorphism group of $B$, and
$\Out^0(B)$ is its identity component as an algebraic group.
Derived equivalence preserves this identity component; see
\cite[Theorem~4.6 and Section~4.2.3]{Rouquier}.
More explicitly, a standard derived equivalence
$F:\Db(\mod B)\to\Db(\mod H)$ identifies
$\Out^0(B)$ with $\Out^0(H)$ by conjugation, as algebraic groups;
see \cite[Theorem~1]{Yek}.
The Picard group $\Pic(B)$ is the group of isomorphism classes of
invertible $B$-$B$-bimodules, with multiplication induced by tensor
product. For a basic algebra, we identify $\Pic(B)$ with $\Out(B)$.
The quotient $\Out(B)/\Out^0(B)$ is finite.
}

{
{Here $\Out(B)=\Out_\k(B)=\Aut_\k(B)/\Inn(B)$, where} $\Inn({B})$ denotes the group of inner automorphisms
of the algebra ${B}$. {In this formula, $\Aut_\k(B)$ denotes algebra automorphisms.
Appendix~\ref{app:picard-basic} explains its identification, modulo inner automorphisms, with
natural-isomorphism classes of $\k$-linear category automorphisms of $B$.}
{Connectedness} is taken in the Zariski topology.

We regard $\Out_\k({B})$ as a subgroup of $\DPic_\k({B})$
by sending each outer automorphism class to the class
of its induced standard derived autoequivalence.
This assignment is well-defined because every inner
automorphism induces a functor naturally isomorphic
to the identity functor.
Conversely, an automorphism inducing such a functor
is inner, so the resulting homomorphism is injective.
By \cite[Theorem~1]{Yek}, the subgroup $\Out^0({B})$
is normal in $\DPic_\k({B})$.
}
\end{rmk}
\endgroup
}

\section{The classification in tree type}\label{sec:tree}

\begingroup
Let $B$ be piecewise hereditary of type $\De$, where $\De$ is
$\widetilde D_r$ or one of the three $\widetilde E$ types. Fix
$\ph\in\Aut(B)$ and $n\geq1$. The underlying graph of $\De$ is a tree. 

Put
\[
 \Aut_0(\De)=
 \{\si\in\Aut(\De)\mid
       \si\text{ fixes a vertex of }\De\}.
\]
This is a subgroup of $\Aut(\De)$.
An orientation $Q_\De$ of $\De$ is called
\emph{admissible} if
\[
 \Aut(Q_\De)=\Aut_0(\De),
\]
where we regard each quiver automorphism as an
automorphism of the underlying tree.
Choose such an orientation by \cite[Lemma~4.1]{Asa},
and put $H_\De=\k Q_\De$.
The hereditary reduction in \cite[Lemma~5.4 and Proposition~5.5]{Asa}
gives
\[
 T_\ph^n(B)\dereq T_{\k\si}^n(H_\De)
\]
for some $\si\in\Aut(Q_\De)$.
Here $\k\si$ is the automorphism of the path category that sends each
vertex and each arrow to its image under $\si$.
\endgroup

The Main Theorem of \cite{Asa} says that the complete invariant is
\begin{equation}\label{eq:treetype}
 (\De,n,\ovl\si).
\end{equation}
Here $\ovl\si$ denotes the conjugacy class of $\si$
in $\Aut(\De)$, the automorphism group of the
unoriented tree.
{For each fixed $\De$ and positive integer $n$,} the set of all possible
third components is exactly
\[
 \{\ovl\si\mid\si\in\Aut_0(\De)\}.
\]
Indeed, the hereditary reduction above shows that the
third component of every extension considered here
belongs to this set.
Conversely, every $\si\in\Aut_0(\De)=\Aut(Q_\De)$
defines an automorphism $\k\si$ of $H_\De$, and
\[
 T_{\k\si}^n(H_\De)
\]
has type $(\De,n,\ovl\si)$.
Thus every class in the displayed set occurs.
There is no further connected scalar parameter. Indeed
$\Out^0(H_\De)=1$ by \cite[Proposition~1.7]{MY}. Under
\eqref{eq:mesh}, the realizable conjugacy classes of derived
autoequivalences are the classes in \eqref{eq:treetype}: restriction
gives the action on the repetition quiver, and passage to its
$\ta$-orbits gives $\si$. 
Conjugacy here is taken in $\Aut(\De)$.
{More explicitly, for all $\si,\si'\in\Aut(Q_\De)$,}
\[
 \ovl\si=\ovl{\si'}
 \quad\Longleftrightarrow\quad
 \si'=\rho\si\rho^{-1}
 \quad\text{for some }\rho\in\Aut(\De).
\]
The automorphism $\rho$ need not preserve the
orientation of $Q_\De$.
This case is handled by the equivariant tilting
construction in the proof of the Main Theorem of
\cite{Asa}, which allows a change of hereditary
orientation.
This proves \thmref{thm:main} for the tree types. Moreover, by
\cite[Proposition~8.1]{Asa}, if two twisted extensions are stably
equivalent and one has tree type, then the other has the same underlying
tree type. Hence it remains to consider the types $\Aaff pq$.

\section{Recovering the affine twist from a full component}\label{sec:mesh}

Fix $H=\k Q$ of type $\Aaff pq$, $\ph\in\Aut(H)$, and $n\geq1$.
Write $E=T_\ph^n(H)$ and $g=\widehat\ph\nu_H^n$.
Let $\pi:\widehat H\to E$ be the canonical covering functor.
The pull-up functor
\[
 \pi\up:\Mod E\longrightarrow\Mod\widehat H
\]
is the restriction functor along $\pi$.
Thus $\pi\up N=N\circ\pi$ for each right $E$-module $N$.
Its left adjoint is the push-down functor
\[
 \pi\down=\blank\otimes_{\widehat H}E:
 \Mod\widehat H\longrightarrow\Mod E,
\]
where $E$ is regarded as a $\widehat H$-$E$-bimodule
via $\pi$.
For each object $x$ of $\widehat H$, we have
\[
 \pi\down\bigl(\widehat H(\blank,x)\bigr)
   \mathrel{{\cong}} E(\blank,\pi x).
\]
Since $\pi\down$ preserves direct sums and direct summands,
it sends projective modules to projective modules.
Moreover, every finitely generated module is a quotient
of a finite direct sum of representable modules.
Since $\pi\down$ is right exact, it also preserves
finitely generated modules.
Thus it restricts to a functor
\[
 \pi\down:\mod\widehat H\longrightarrow\mod E.
\]
By the covering property, $E$, regarded as a left
$\widehat H$-module via $\pi$, is projective and hence flat.
Therefore $\pi\down$ is exact.
It follows that $\pi\down$ induces a triangle functor
\[
 \pi_*:\umod\widehat H\longrightarrow\umod E,
\]
which we call the \emph{stable pushdown functor}.

\begingroup
Let $R:\Db(\mod H)\to\umod\widehat H$
be Happel's equivalence, chosen so that its restriction
to $\mod H$ is the canonical embedding.
We use the composite functor
\[
 P=\pi_*R:\Db(\mod H)\longrightarrow\umod E.
\]
Write $\Ph$ for the derived autoequivalence induced by
$\ph$, with the convention fixed in Section~1.
Let $\widehat{\Ph}$ and $\Ga$ be the autoequivalences
of $\umod\widehat H$ induced by $\widehat\ph$ and $g$,
respectively.
Our next aim is to identify the autoequivalence
corresponding to $\Ga$ under $R$.
We first prove the following lemma.

\begin{lem}\label{lem:standard-identity}
Let $\calP$ be the full subcategory of $\Db(\mod H)$
consisting of finitely generated projective $H$-modules
regarded as complexes concentrated in degree zero,
and let $\iota:\calP\to\Db(\mod H)$ be the inclusion.
Let $F$ be a standard autoequivalence of $\Db(\mod H)$.
If there is a natural isomorphism
$\eta:F\iota\Longrightarrow\iota$, then
$F\mathrel{{\cong}}\mathrm{Id}$.
\end{lem}

\begin{proof}
Choose a two-sided tilting complex $T$ and a natural
isomorphism
\[
 \xi:\blank\otimes_H^{\mathbf L}T\Longrightarrow F.
\]
The components of $\et$ and $\xi$ at $H$ give isomorphisms
\[
 T\xrightarrow{\xi_H}F(H)\xrightarrow{\eta_H}H
\]
in $\Db(\mod H)$.
Hence $H^i(T)=0$ {for every $i\in\bbZ\setminus\{0\}$,} and therefore
$T\mathrel{{\cong}} H^0(T)$ in the derived category of
$H$-$H$-bimodules.
Set
\[
 \theta=H^0(\eta_H\xi_H):H^0(T)\longrightarrow H.
\]
This is an isomorphism of right $H$-modules.
We show that it is also left $H$-linear.

{For each $a\in H$,} let $\ell_a:H\to H$ be left
multiplication by $a$.
The morphism $\ell_a$ is right $H$-linear, and
$\ell_a\otimes_H^{\mathbf L}T$ is left multiplication
by $a$ on $T$.
The naturality of $\xi$ and $\eta$ therefore gives
\[
\begin{tikzcd}[column sep=large,row sep=large]
 H^0(T) \arrow[r,"a\cdot\blank"]
        \arrow[d,"\theta"'] &
 H^0(T) \arrow[d,"\theta"] \\
 H \arrow[r,"\ell_a"'] & H.
\end{tikzcd}
\]
Thus $\theta(at)=a\theta(t)$ for all $a\in H$
and $t\in H^0(T)$.
Consequently, $\theta$ is an isomorphism of
$H$-$H$-bimodules, and $T\mathrel{{\cong}} H$ in their
derived category.
It follows that
\[
 F\mathrel{{\cong}}\blank\otimes_H^{\mathbf L}T
  \mathrel{{\cong}}\blank\otimes_H^{\mathbf L}H
  \mathrel{{\cong}}\mathrm{Id}.
\]
\end{proof}

\begin{lem}\label{lem:Happel-actions}
The following statements hold.
\begin{enumerate}
\item\label{item:actions-standard} There is a natural isomorphism
$R\Ph\mathrel{{\cong}}\widehat{\Ph}R$.
Thus the following diagram commutes up to natural
isomorphism:
\[
\begin{tikzcd}[column sep=large,row sep=large]
 \Db(\mod H) \arrow[r,"R"] \arrow[d,"\Ph"'] &
 \umod\widehat H \arrow[d,"\widehat{\Ph}"] \\
 \Db(\mod H) \arrow[r,"R"'] & \umod\widehat H.
\end{tikzcd}
\]
\item\label{item:actions-generator} Set
\begin{equation}\label{eq:generator}
 G=\Ph\ta^n[2n].
\end{equation}
Then there is a natural isomorphism
$RG\mathrel{{\cong}}\Ga R$.
\end{enumerate}
\end{lem}

\begin{proof}
{(\ref{item:actions-standard})} Let $R^-$ and $\Ph^-$ be quasi-inverses of
$R$ and $\Ph$, respectively, and set
\[
 F=\Ph^-R^-\widehat{\Ph}R.
\]
By \cite[Lemma~2.7]{Asa}, the restrictions of
$\widehat{\Ph}R$ and $R\Ph$ to $\mod H$
are naturally isomorphic.
Thus we obtain natural isomorphisms
\[
 \begin{aligned}
 F\iota
 &=\Ph^-R^-\widehat{\Ph}R\iota\\
 &\mathrel{{\cong}}\Ph^-R^-R\Ph\iota\\
 &\mathrel{{\cong}}\Ph^-\Ph\iota\\
 &\mathrel{{\cong}}\iota.
 \end{aligned}
\]
Their composite gives a natural isomorphism
$\et:F\iota\Longrightarrow\iota$.
Thus the hypothesis concerning $\et$ in
\lemref{lem:standard-identity} is satisfied.
By \cite[Theorem~1.8]{MY}, $F$ is standard.
Applying \lemref{lem:standard-identity}, we obtain
$F\mathrel{{\cong}}\mathrm{Id}$, and hence
$R\Ph\mathrel{{\cong}}\widehat{\Ph}R$.

{(\ref{item:actions-generator})} Let $N_{\widehat H}$ be the stable Nakayama
functor.
As recalled after \eqref{eq:Happel}, there is a
natural isomorphism
\[
 R\ta[2]\mathrel{{\cong}} N_{\widehat H}R.
\]
Since $g=\widehat\ph\nu_H^n$, its induced functor
satisfies $\Ga\mathrel{{\cong}}\widehat{\Ph}N_{\widehat H}^n$.
Using {(\ref{item:actions-standard})}, we obtain
\[
 \begin{aligned}
 RG
 &=R\Ph\ta^n[2n]\\
 &\mathrel{{\cong}}\widehat{\Ph}R\ta^n[2n]\\
 &\mathrel{{\cong}}\widehat{\Ph}N_{\widehat H}^nR\\
 &\mathrel{{\cong}}\Ga R.
 \end{aligned}
\]
This proves {(\ref{item:actions-generator})}.
\end{proof}

We now recall the full subcategory
$\calB\subseteq\Db(\mod H)$ of \prpref{prop:MY}.
It has one representative of each isomorphism class
of indecomposable objects in the transjective component
containing the projective modules.
These objects are the preprojective modules in degree
zero and the preinjective modules shifted by $[-1]$.
Both $\ta$ and $\Ph$ preserve this component.
Their restrictions to the chosen representatives are
understood up to natural isomorphism.
By \eqref{eq:generator}, $G$ sends this component to
its shift by $[2n]$.
\endgroup

{In formulas involving actual functors, negative powers of $G$, $\Ph$, and $\ta$ denote iterates of chosen quasi-inverses.}

\begin{lem}\label{lem:vanish}
For any $X,Y\in\calB$ and any nonzero integer $j$,
we have
\[
{ \Db(\mod H)(G^jX,Y)}=0.
\]
\end{lem}
\begin{proof}
\begingroup
Because $\Ph$ and $\ta$ preserve the chosen
transjective component, we can write
$\Ph^j\ta^{nj}X\mathrel{{\cong}} M[a]$ and $Y\mathrel{{\cong}} L[b]$,
where $M,L$ are $H$-modules and $a,b\in\{0,-1\}$.
Thus
\[
{ \Db(\mod H)(G^jX,Y)}
 \mathrel{{\cong}}\Ext_H^d(M,L),\qquad d=b-a-2nj.
\]
The possible degrees give the following reduction:
\[
 \begin{array}{c|c|c}
 \text{case}&\text{degree}&\text{reason for vanishing}\\ \hline
 j>0&d\leq-1&\text{negative degree between modules}\\
 j<0,\ d\geq2&d\geq2&\text{$H$ is hereditary}\\
 j<0,\ d=1&n=1,\ j=-1,\ a=0,\ b=-1
   &\text{considered below}.
 \end{array}
\]
In the last case, $M$ is preprojective and $L$ is preinjective.
If $M$ is projective, then $\Ext_H^1(M,L)=0$.
Otherwise Auslander--Reiten duality and the absence of morphisms from
preinjectives to preprojectives give
\[
 D\Ext_H^1(M,L)\mathrel{{\cong}}\ovl{\Hom}_H(L,\ta M)=0.
\]
Here $\ovl{\Hom}_H$ denotes morphisms modulo those factoring through
injective modules; it is zero because already
$\Hom_H(L,\ta M)=0$.
\endgroup
\end{proof}

\begin{lem}\label{lem:full}
The restriction $P|_\calB$ is an equivalence onto the full subcategory
$\calC_\ph$ on one transjective component of $\umod E$. The stable
Auslander--Reiten quiver of $E$ has exactly $2n$ transjective components,
each isomorphic to $\bbZ Q$. All its other components are
$\ta$-periodic. The functor $\Om$ permutes the transjective
components transitively.
\end{lem}
\begin{proof}
\begingroup
{For all $X,Y\in\Db(\mod H)$,} the covering formula, followed by
Happel's equivalence, gives
\begin{equation}\label{eq:stablecover}
 \begin{aligned}
 \uHom_E(PX,PY)
 &\mathrel{{\cong}}\Ds_{j\in\bbZ}
       \uHom_{\widehat H}(g^jRX,RY)\\
 &\mathrel{{\cong}}\Ds_{j\in\bbZ}
      { \Db(\mod H)(G^jX,Y)}.
 \end{aligned}
\end{equation}
In the first line, $g^j$ denotes the induced action on modules.
\endgroup
We show that the covering isomorphism on module
morphisms induces the first isomorphism in
\eqref{eq:stablecover}.
For $\calA=\widehat H$ or $E$, and for any
$U,V\in\mod\calA$, let $\calI_{\calA}(U,V)$
be the subspace of morphisms that factor through
a finitely generated projective $\calA$-module.
For any $M,N\in\mod\widehat H$, the canonical
covering isomorphism is
\[
 \Psi_{M,N}:
 \Ds_{j\in\bbZ}\Hom_{\widehat H}(g^jM,N)
 \xrightarrow{\ \cong\ }
 \Hom_E(\pi\down M,\pi\down N).
\]
We claim that
\[
 \Psi_{M,N}^{-1}
 \bigl(\calI_E(\pi\down M,\pi\down N)\bigr)
 =\Ds_{j\in\bbZ}\calI_{\widehat H}(g^jM,N).
\]

To prove the inclusion from left to right, take any
$f\in\calI_E(\pi\down M,\pi\down N)$.
Every finitely generated projective $E$-module is
a direct summand of a finite direct sum of modules
$E(\blank,\pi x)$, and
\[
 E(\blank,\pi x)
 \cong\pi\down\bigl(\widehat H(\blank,x)\bigr).
\]
Thus we can choose a finitely generated projective
$\widehat H$-module $Q$ and a factorization
\[
\begin{tikzcd}[column sep=large]
 \pi\down M \arrow[r,"u"]
   \arrow[rr,bend left=35,"f"] &
 \pi\down Q \arrow[r,"v"] & \pi\down N.
\end{tikzcd}
\]
Write
\[
 \Psi_{M,Q}^{-1}(u)=(u_r)_{r\in\bbZ},
 \qquad
 \Psi_{Q,N}^{-1}(v)=(v_s)_{s\in\bbZ},
\]
where
\[
 u_r:g^rM\to Q,
 \qquad v_s:g^sQ\to N.
\]
{Only finitely many of the morphisms $u_r$ and $v_s$ are nonzero.}
By the composition rule for the canonical covering
maps, for each $t\in\bbZ$ we have
\[
 \bigl(\Psi_{M,N}^{-1}(f)\bigr)_t
 =\sum_{r+s=t}v_s\circ g^s(u_r).
\]
{For each pair $(r,s)\in\bbZ^2$ with $r+s=t$,} the corresponding summand
has the factorization
\[
\begin{tikzcd}[column sep=large]
 g^{r+s}M \arrow[r,"g^s(u_r)"]
   \arrow[rr,bend left=35,"v_s\circ g^s(u_r)"] &
 g^sQ \arrow[r,"v_s"] & N.
\end{tikzcd}
\]
Since $g^sQ$ is projective and the sum is finite,
\[
 \bigl(\Psi_{M,N}^{-1}(f)\bigr)_t
 \in\calI_{\widehat H}(g^tM,N){.}
\]
This proves the first inclusion.

Conversely, take any $j\in\bbZ$ and any morphism
$h\in\calI_{\widehat H}(g^jM,N)$.
Choose a factorization
\[
 g^jM\xrightarrow{a}Q\xrightarrow{b}N,
 \qquad h=ba,
\]
with $Q$ finitely generated projective.
Applying $\pi\down$ and using the canonical
isomorphism $\pi\down g^jM\cong\pi\down M$,
we obtain a factorization of its image under
$\Psi_{M,N}$ through the projective module
$\pi\down Q$.
Taking finite sums proves the reverse inclusion.

The claimed equality now allows us to take quotients
by these subspaces, and gives
\[
 \Ds_{j\in\bbZ}\uHom_{\widehat H}(g^jM,N)
 \cong\uHom_E(\pi\down M,\pi\down N).
\]
Applying this to modules representing $RX$ and $RY$
proves the first isomorphism in
\eqref{eq:stablecover}.

For the second isomorphism, \lemref{lem:Happel-actions}{(\ref{item:actions-generator})}
gives a natural isomorphism
$RG^j\cong\Ga^jR$ for each $j\in\bbZ$.
Here $\Ga^j$ is the functor denoted by $g^j$
on stable modules in \eqref{eq:stablecover}.
Since $R$ is fully faithful, for any
$X,Y\in\Db(\mod H)$ and any $j\in\bbZ$, we obtain
\[
 \begin{aligned}
 \uHom_{\widehat H}(g^jRX,RY)
 &\cong\uHom_{\widehat H}(RG^jX,RY)\\
 &\cong{\Db(\mod H)(G^jX,Y).}
 \end{aligned}
\]
Taking the direct sum over all $j\in\bbZ$ proves
the second isomorphism in \eqref{eq:stablecover}.

\begingroup
{For all $X,Y\in\calB$,} \lemref{lem:vanish} leaves only the summand
$j=0$ in \eqref{eq:stablecover}. Thus pushdown induces isomorphisms
\[
{ \calB(X,Y)\xrightarrow{\ \sim\ }\uHom_E(PX,PY).}
\]
This is full faithfulness on all morphisms, including morphisms that
are not irreducible.
\endgroup

{For a right module $M$ over a small $\k$-category $A$, its \emph{support} is the set $\{x\in A_0\mid
M(x)\ne0\}$, where $A_0$ is the object set. The category $A$ is \emph{locally support-finite} if, for each
$x\in A_0$, the union of the supports of indecomposable finite-dimensional $A$-modules $M$ with $M(x)\ne0$ is
finite.}
The repetitive category of a hereditary algebra is locally
support-finite by \cite[Lemma~6.1]{Asa}. Its pushdown is therefore essentially surjective,
and the component calculation in \cite[Lemma~6.2]{Asa} applies to the
entire stable Auslander--Reiten quiver. \begingroup
The transjective components {in $\Db(\mod H)$} are the successive shifts
$\calB[t]$, indexed by $t\in\bbZ$. These are distinct
Auslander--Reiten components, although each has translation quiver
isomorphic to $\bbZ Q$. The action of $G$ is
\[
 \calB[t]\longmapsto\calB[t+2n].
\]
Consequently the {transjective components} {in $\umod E$} are indexed by
$\bbZ/(2n)\bbZ$, not by the common isomorphism type $\bbZ Q$.
The component calculation also identifies the essential image of
$P|_\calB$ with the component indexed by $0$; this proves essential
surjectivity onto $\calC_\ph$. The remaining hereditary affine
components are tubes and remain $\ta$-periodic. Finally,
\[
 \Om:\ t\bmod 2n\longmapsto t-1\bmod 2n,
\]
so $\Om$ acts transitively on these $2n$ components.
\endgroup
\end{proof}

\begin{lem}[The return functor]\label{lem:return}
\begingroup
Let $N_E:\umod E\to\umod E$ be the stable Nakayama functor.
The functors
\[
 N_E^n P|_\calB,\quad P\Ph^-|_\calB
 :\calB\to\umod E
\]
are naturally isomorphic:
\begin{equation}\label{eq:return}
N_E^n P|\calB\ \cong\ P\Ph^-|\calB.
\end{equation}
In particular, $N_E^n$ preserves $\calC_\ph$ up to isomorphism.
Here $\Ph^-$ denotes a chosen quasi-inverse of $\Ph$, as in
Section~1.
\endgroup
\end{lem}
\begin{proof}
We first check that pushdown commutes with the Nakayama functors by a
natural isomorphism. For this purpose we use the duality on morphism
spaces, as follows. For arbitrary objects $x,y$ of $\widehat H$, the repetitive duality
\[
 D\widehat H(x,y)\mathrel{{\cong}}\widehat H(y,\nu_Hx)
\]
is compatible with $\nu_H$ and with the standard lift $\widehat\ph$;
{on the morphism spaces $DH(y,x)$ between
adjacent layers, the latter uses the inverse dual action.} Summing this
{duality over all powers $g^j$, $j\in\bbZ$, identifies the descended $\nu_H$ with the}
Nakayama functor of the orbit category. More explicitly, reindexing
the finite sums uses
\[
{ D\Ds_{j\in\bbZ}\widehat H(x,g^jy)
 \mathrel{{\cong}}\Ds_{j\in\bbZ}\widehat H(g^jy,\nu_Hx)
 \mathrel{{\cong}}\Ds_{j\in\bbZ}\widehat H(y,g^{-j}\nu_Hx).}
\]
\begingroup
The identification is natural on projectives and extends to modules
by projective presentations. In terms of the two separate functors
introduced above, it says
\[
 N_E\pi_*\mathrel{{\cong}}\pi_*N_{\widehat H},\qquad
 N_EP\mathrel{{\cong}} P\ta[2].
\]
The canonical orbit isomorphism is {$\pi_*\Ga\mathrel{{\cong}}\pi_*$}, or,
after composition with $R$, $PG\mathrel{{\cong}} P$.
The standard lift commutes with $\nu_H$, so these natural
isomorphisms give
\[
 N_E^nP\Ph\mathrel{{\cong}} P\ta^n[2n]\Ph
 \mathrel{{\cong}} P\Ph\ta^n[2n]=PG\mathrel{{\cong}} P.
\]
Composing with a quasi-inverse of $\Ph$ proves
\eqref{eq:return}. Equivalently, the following square commutes up to
the resulting natural isomorphism:
\[
\begin{tikzcd}[column sep=large,row sep=large]
 \calB
   \arrow[r,"P|_\calB"]
   \arrow[d,"\Ph^-|_\calB"'] &
 \calC_\ph
   \arrow[d,"N_E^n|_{\calC_\ph}"] \\
 \calB
   \arrow[r,"P|_\calB"'] &
 \calC_\ph.
\end{tikzcd} 
\]
\endgroup
\end{proof}

\begin{thm}[Hereditary affine criterion]\label{thm:hereditary}
{Let $\ph,\ps\in\Aut(H)$ and let $n,n'\geq1$ be integers. Denote the derived autoequivalences induced
by $\ph$ and $\ps$ by $\Ph$ and $\Ps$, respectively.}
The following are equivalent:
\begin{enumerate}
\item\label{item:hereditary-derived} $T_\ph^n(H)$ and $T_\ps^{n'}(H)$ are derived equivalent;
\item\label{item:hereditary-stable} they are stably equivalent;
\item\label{item:hereditary-type} $n=n'$ and $\Ph,\Ps$ are conjugate in $\DPic_\k(H)$.
\end{enumerate}
\end{thm}
\begin{proof}
\noindent
{(\ref{item:hereditary-derived})} $\Rightarrow$ {(\ref{item:hereditary-stable})}.
By Rickard's theorem \cite{Ric89}, a derived equivalence
between selfinjective algebras induces a triangle equivalence
between their stable module categories.
\begingroup

{(\ref{item:hereditary-stable})} $\Rightarrow$ {(\ref{item:hereditary-type})}.
Put $E_\ph=T_\ph^n(H)$ and $E_\ps=T_\ps^{n'}(H)$, and let
$V:\umod E_\ph\to\umod E_\ps$ be a triangle equivalence as in {(\ref{item:hereditary-stable})}.
The nonperiodic components are exactly the transjective ones, so
\lemref{lem:full} gives $2n=2n'$ and hence $n=n'$.
By the transitive action of $\Om$, we may compose $V$ with a power
of $\Om$ on the target so that it takes $\calC_\ph$ to
$\calC_\ps$, up to isomorphism.

{A \emph{Serre functor} on a $\k$-linear category $\calD$ with finite-dimensional morphism
spaces is an autoequivalence $S$ equipped with bifunctorial isomorphisms $D\calD(X,Y)\cong\calD(Y,SX)$ for
all $X,Y\in\calD$.}
Let $S_\ph=\Om N_{E_\ph}$ and
$S_\ps=\Om N_{E_\ps}$ be the stable Serre functors.
Since $V$ is a triangle equivalence, it commutes with
suspension and hence with its quasi-inverse $\Om$.
It also intertwines the Serre functors.
Thus we have natural isomorphisms
\[
 V\Om\cong\Om V,
 \qquad VS_\ph\cong S_\ps V.
\]
Using the definitions of $S_\ph$ and $S_\ps$, we obtain
\[
 \Om V N_{E_\ph}
 \cong V\Om N_{E_\ph}
 =VS_\ph
 \cong S_\ps V
 =\Om N_{E_\ps}V.
\]
Applying a quasi-inverse of $\Om$ gives a natural
isomorphism
\[
 VN_{E_\ph}\cong N_{E_\ps}V.
\]
Denote the equivalences in \lemref{lem:full} by
$P_\ph:\calB\to\calC_\ph$ and $P_\ps:\calB\to\calC_\ps$,
and choose a quasi-inverse $P_\ps^-$ of $P_\ps$.
{Choose representatives in $\calC_\ps$ and isomorphisms from the images under $V$ to these
representatives. As in \prpref{prop:MY}, use them to regard $V|_{\calC_\ph}$ as an equivalence
$\calC_\ph\to\calC_\ps$.}
Define
\[
 U=P_\ps^-\bigl(V|_{\calC_\ph}\bigr)P_\ph:\calB\to\calB.
\]
Using the return isomorphisms, we obtain the following square, which
commutes up to natural isomorphism:
\[
\begin{tikzcd}[column sep=large,row sep=large]
 \calB
   \arrow[r,"U"]
   \arrow[d,"\Ph^-|_\calB"'] &
 \calB
   \arrow[d,"\Ps^-|_\calB"] \\
 \calB
   \arrow[r,"U"'] &
 \calB.
\end{tikzcd}
\]
By \prpref{prop:MY}, there exists a standard derived
autoequivalence
\[
 \widetilde U:\Db(\mod H)\to\Db(\mod H)
\]
which preserves the transjective component containing
the indecomposable projective modules and whose
restriction to $\calB$ is naturally isomorphic to $U$.
Here the restriction to $\calB$ is understood by choosing
representatives in $\calB$, as in \prpref{prop:MY}.
Both $\Ph$ and $\Ps$ also preserve this component.
Thus $\widetilde U\Ph^-$ and $\Ps^-\widetilde U$
preserve this component, and the preceding commutative
diagram shows that their restrictions to $\calB$
are naturally isomorphic.
Under the second isomorphism in \eqref{eq:mesh},
the classes of these two functors have the same image:
their restrictions to $\calB$ are naturally isomorphic,
and the integer associated with each functor is zero
because each preserves the chosen component.
Since this isomorphism is injective, the two classes
are equal. Hence there is a natural isomorphism
\[
 \widetilde U\Ph^-\cong\Ps^-\widetilde U.
\]
Equivalently, $\Ph$ and $\Ps$ are conjugate in $\DPic_\k(H)$.
This proves {(\ref{item:hereditary-type})}.
\endgroup

\begingroup
{(\ref{item:hereditary-type})} $\Rightarrow$ {(\ref{item:hereditary-derived})}.  Suppose
{(\ref{item:hereditary-type})} holds.
Choose a standard derived autoequivalence $U:\Db(\mod H)\to\Db(\mod H)$
and a natural isomorphism $\Ph U\mathrel{{\cong}} U\Ps$.
{Write $H_0$ for the object set of $H$. For each $x\in H_0$, put}
\[
 P_x:=H(\blank,x).
\]
Let $\calE$ be the full subcategory of $\Db(\mod H)$
whose objects are
\[
 U(P_x)\qquad (x\in H_0).
\]
Thus, for any $x,y\in H_0$, we have
\[
 \calE(U(P_x),U(P_y))
 =\Db(\mod H)(U(P_x),U(P_y)).
\]
 It is a tilting spectroid and
the restriction of $U$ to the full subcategory on the
{chosen projectives $P_x$ $(x\in H_0)$} is an equivalence onto $\calE$.
Let $\eta:\Ph U\To U\Ps$ be the chosen natural
isomorphism. For each object $x$ of $H$, let
\[
 \kappa_x:\Ps(P_x)\to P_{\ps(x)}
\]
be the canonical isomorphism induced by $\ps$, and put
\[
 \theta_x:
 \Ph(U(P_x))
 \xrightarrow{\eta_{P_x}} U(\Ps(P_x))
 \xrightarrow{U(\kappa_x)} U(P_{\ps(x)}).
\]
Since $U(P_{\ps(x)})$ is an object of $\calE$,
these isomorphisms show that $\calE$ is stable under
$\Ph$ up to isomorphisms.

We also describe the action on morphisms.
For each morphism $f:x\to y$ in $H$, put
$P_f:=H(\blank,f):P_x\to P_y$.
By the naturality of $\eta$ and of the canonical
isomorphisms $\kappa_x$, the following diagram commutes:
\[
\begin{tikzcd}[column sep=large,row sep=large]
 \Ph(U(P_x))
   \arrow[r,"\Ph(U(P_f))"]
   \arrow[d,"\theta_x"'] &
 \Ph(U(P_y))
   \arrow[d,"\theta_y"] \\
 U(P_{\ps(x)})
   \arrow[r,"U(P_{\ps(f)})"'] &
 U(P_{\ps(y)}).
\end{tikzcd}
\]
Thus, after identifying $\Ph(U(P_x))$ with
$U(P_{\ps(x)})$ by $\theta_x$, the action on each
morphism $U(P_f)$ is given by
\[
 \theta_y\circ\Ph(U(P_f))\circ\theta_x^{-1}
 =U(P_{\ps(f)}).
\]
By the Yoneda lemma and the full faithfulness of $U$,
every morphism in $\calE$ has the form $U(P_f)$
for a unique morphism $f$ in $H$.
Hence the action transported to the projective
spectroid of $H$ agrees with the action induced by
$\ps$, on both objects and morphisms.
Thus \prpref{prop:tilting} gives the desired derived equivalence for
{every integer $n\geq1$. }
Appendix~\ref{app:descent} explains how the natural
isomorphism $\eta:\Ph U\To U\Ps$ is used to construct
the equivariant equivalence for the infinite cyclic
actions and to obtain the derived equivalence of
the orbit algebras.
The isomorphisms $\theta_x$ above are sufficient;
we do not require the equalities
$\Ph(U(P_x))=U(P_{\ps(x)})$.
\endgroup
\end{proof}

\begin{prp}[Recovering the hereditary type]\label{prop:recover}
The stable category of a hereditary affine twisted extension
determines $\De$ and $n$.
\end{prp}

\begin{proof}
\begingroup
We first recover $n$.
By \lemref{lem:full}, there are exactly $2n$
nonperiodic components in the stable Auslander--Reiten
quiver.
The same assertion holds in the affine tree types:
the proof uses the vanishing between the shifted
transjective components, and all the remaining
components are $\ta$-periodic.
Thus the number of nonperiodic components determines $n$.

We next recover $\De$.
Choose one nonperiodic component $\Gamma$ and write
\[
 \Gamma\cong\bbZ Q,
\]
where $H=\k Q$ is the hereditary algebra used in the
extension.
Identifying $\Gamma$ with $\bbZ Q$ via this isomorphism,
consider the projection
\[
 \pi_0:\Gamma_0\to Q_0,\qquad (i,v)\mapsto v,
\]
where $\Gamma_0$ and $Q_0$ denote the vertex sets.
Since $\ta(i,v)=(i-1,v)$, the fibers of $\pi_0$ are
precisely the $\ta$-orbits.
Thus $\pi_0$ induces a bijection from the set of
$\ta$-orbits of vertices of $\Gamma$ to $Q_0$.
For each arrow $u\to v$ in $Q$, the repetition quiver
has arrows
\[
 (i,u)\to(i,v),\qquad
 (i,v)\to(i+1,u)
 \qquad(i\in\bbZ),
\]
and every arrow of $\Gamma$ has one of these forms.
Consequently, two distinct vertices $u,v$ of $Q$
are adjacent in the underlying unoriented graph
if and only if an arrow of $\Gamma$ joins
$\pi_0^{-1}(u)$ and $\pi_0^{-1}(v)$ in either direction.
Hence we recover adjacency in the underlying graph
from $\Gamma$ and $\ta$.
We distinguish the following three cases.

\smallskip
\noindent
\setcounter{proofstep}{0}
{\proofstep{case:recover-tree}} The underlying graph of $Q$ is a tree.
The adjacency described above recovers this tree
from $\Gamma$.
It therefore determines which of the types
$\widetilde D_r$, $\widetilde E_6$,
$\widetilde E_7$, and $\widetilde E_8$ occurs.

\smallskip
\noindent
{\proofstep{case:recover-cycle}} The underlying graph of $Q$ is a cycle of
length at least three.
The adjacency described above recovers the cycle.
To determine its type, we also recover the numbers
$p$ and $q$ of arrows in its two cyclic directions.
Choose one of these directions, and let $p$ be the
number of arrows pointing in this direction and
$q$ the number pointing in the opposite direction.

Recall that, for each arrow $a:u\to v$ of $Q$
and each integer $i$, the repetition quiver has arrows
\[
 (i,u)\longrightarrow(i,v),
 \qquad
 (i,v)\longrightarrow(i+1,u),
\]
and $\ta(i,u)=(i-1,u)$.
Starting at a vertex $x$ of $\Gamma$, follow arrows
whose successive $\ta$-orbits go once around the
cycle in the chosen direction.
At each step there is a unique such arrow.
The first coordinate stays unchanged at each of the
$p$ steps agreeing with the orientation of $Q$,
and increases by one at each of the other $q$ steps.
Hence the resulting path ends at $\ta^{-q}x$.
Going around the cycle in the opposite direction
similarly gives a path ending at $\ta^{-p}x$.
Thus the two lifted paths have the form
\[
 x=x_0\longrightarrow x_1\longrightarrow\cdots
   \longrightarrow x_{p+q}=\ta^{-q}x,
\]
\[
 x=y_0\longrightarrow y_1\longrightarrow\cdots
   \longrightarrow y_{p+q}=\ta^{-p}x.
\]
These paths are determined by the translation quiver
$\Gamma$ once $x$ and a cyclic direction are chosen.
Consequently, $\Gamma$ determines the unordered pair
$\{p,q\}$.
Our convention $p\geq q$ then determines the type
$\widetilde A_{p,q}$.

\smallskip
\noindent
{\proofstep{case:recover-Kronecker}} $Q$ is the Kronecker quiver.
In this case $\Gamma$ has exactly two $\ta$-orbits.
For a vertex $x$ in either orbit, there are two
parallel arrows from $x$ to a vertex in the other
orbit.
These properties distinguish the Kronecker case,
which has type $\widetilde A_{1,1}$.

\smallskip
A triangle equivalence of stable categories preserves
the stable Auslander--Reiten quiver, including its
translation and arrow multiplicities.
It therefore preserves the data used above.
Thus the stable category determines both $\De$ and $n$.
\par\endgroup
\end{proof}

\begingroup
For comparison, if $\La=T_\ph^n(B)$ and $\De$ is known,
the folding number can also be read from the
Grothendieck group $K_0(\mod\La)$.
We write $K_0(\La):=K_0(\mod\La)$ and denote by
$|\De|$ the number of vertices of $\De$.
Then
\endgroup
\begin{equation}\label{eq:n}
 n=\frac{\rank K_0(\La)}{|\De|},
 \qquad
 |\De|=
 \begin{cases}
 p+q & \text{if }\De=\Aaff pq,\\
 r+1 & \text{if }\De=\widetilde D_r
                 \text{ or }\widetilde E_r.
 \end{cases}
\end{equation}

\section{Piecewise hereditary reduction in type \texorpdfstring{$\widetilde A$}{A}}\label{sec:reduction}

\begingroup
\begin{rmk}[Choice of the hereditary quiver]\label{rmk:hereditary-models}

In this section we reduce the case of type $\widetilde A$ to hereditary algebras.
For this purpose, we choose the orientation of the
underlying cycle so that the rotations needed below
preserve the directions of the arrows and hence define
quiver automorphisms. 
When $\gcd(p,q)>1$, an orientation with a unique source
does not realize all of them. We use the following models:
\begin{itemize}
\item If $p>q$, put $d=\gcd(p,q)$ and orient the cycle by the repeated
word $(+^{p/d}-^{q/d})^d$. Its quiver automorphism group is cyclic of order $d$,
and we denote it by $C_d$.
\item If $p=q=m>1$, use the $2m$-cycle with alternating arrow directions. 
Its quiver automorphism group is generated by the rotation
through two vertices and a reflection fixing a vertex.
These automorphisms preserve the sources and the sinks.
The group is dihedral of order $2m$, and we denote it by
$D_{2m}$.
\item For $(p,q)=(1,1)$, use the two-arrow Kronecker quiver.
\end{itemize}
\begingroup
Here $+$ and $-$ refer to the two directions around a fixed traversal
of the cycle.
All orientations with the same two orientation numbers give the same
hereditary derived type. In the remainder of this section, $Q$ is the
chosen quiver and $H=\k Q$.

\endgroup
\end{rmk}
\endgroup

{
For the chosen algebra $H=\k Q$, we have
$\Out^0(H)\subseteq\Out_\k(H)=\Pic(H)$,
since $H$ is basic (see \rmkref{rmk:outer-groups}).
}
\begin{lem}\label{lem:finite}
\begingroup
Let $q_H:\DPic_\k(H)\to\DPic_\k(H)/\Out^0(H)$ be the quotient
homomorphism. For the {models in \rmkref{rmk:hereditary-models}}, every finite subgroup of
$\DPic_\k(H)/\Out^0(H)$ is contained in $q_H(\Pic(H))$.
\endgroup
\end{lem}
\begin{proof}
{Put $d:=\gcd(p,q)$.}
First suppose $p+q\geq3$.
The projection $(i,v)\mapsto v$ identifies the $\ta$-orbits
of vertices of $\bbZ Q$ with the vertices of the underlying
unoriented cycle of $Q$.
Every automorphism of $\bbZ Q$ commuting with $\ta$
therefore induces a symmetry of this cycle.
We determine which symmetries lift to such automorphisms
and describe their lifts.
We first construct a lift of the clockwise rotation
by one vertex.
Label the vertices of the underlying cycle clockwise as
$v_0,\ldots,v_{p+q-1}$, with indices taken modulo $p+q$.
Choose this direction so that $p$ arrows of $Q$ point
clockwise and $q$ arrows point counterclockwise.
Recall that $\ta(i,v_j)=(i-1,v_j)$ in $\bbZ Q$.
{For each $i\in\bbZ$ and each $j\in\bbZ/(p+q)\bbZ$,} the arrow in $\bbZ Q$
which starts at $(i,v_j)$ and follows the cycle clockwise
is
\[
\begin{cases}
 (i,v_j)\to(i,v_{j+1}),
   &\text{if }v_j\to v_{j+1}\text{ in }Q,\\
 (i,v_j)\to(i+1,v_{j+1}),
   &\text{if }v_{j+1}\to v_j\text{ in }Q.
\end{cases}
\]
Sending each vertex to the endpoint of this arrow defines
an automorphism of the translation quiver $\bbZ Q$.
Let $\ro$ be its inverse. {Explicitly, for all $i\in\bbZ$ and $j\in\bbZ/(p+q)\bbZ$,}
\[
\ro(i,v_{j+1})=
\begin{cases}
 (i,v_j),&\text{if }v_j\to v_{j+1}\text{ in }Q,\\
 (i-1,v_j),&\text{if }v_{j+1}\to v_j\text{ in }Q.
\end{cases}
\]
Then $\ro$ and $\ta$ commute, and their relation is
\[
 \ro^{p+q}=\ta^q.
\]
{Indeed, a full clockwise traversal starting at any vertex $w$ of $\bbZ Q$ has endpoint $\ta^{-q}w$,} as in
\prpref{prop:recover}. A lift of a fixed cycle symmetry is determined up to a
power of $\ta$. A reversing symmetry can have a commuting lift only
{if the translations $\ta^{-q}$ and $\ta^{-p}$ obtained by traversing the cycle in the two directions
agree, that is, if $p=q$.}

Put $x=\ro$ and $y=\ta\ro^{-1}$. The orientation-preserving group is
\begin{equation}\label{eq:lattice}
 L(p,q)=\langle x,y\mid xy=yx,\ x^p=y^q\rangle
       \mathrel{{\cong}}\bbZ \ds{\bbZ/d\bbZ}.
\end{equation}
If $p>q$, its finite-order elements form precisely the subgroup
generated by
\begin{equation}\label{eq:rotation}
 r=x^{p/d}\,y^{-q/d}
   =\ro^{(p+q)/d}\,\ta^{-q/d}.
\end{equation}
The element $r$ has order $d$, and the subgroup
$\langle r\rangle$ is the rotation group of the chosen
periodic orientation.
The component shift supplies an additional infinite cyclic factor in
$\DPic(H)/\Out^0(H)$, by \cite[Theorem~0.1]{MY}. Thus every finite
subgroup lies in the indicated $C_d$.

Suppose that $p=q=m>1$.
For the alternating quiver $Q$, choose the clockwise
labelling $v_0,\ldots,v_{2m-1}$ so that the vertices
with even indices are sources and those with odd
indices are sinks. Indices are taken modulo $2m$.
The reflection
\[
 v_j\longmapsto v_{-j}
\]
preserves the directions of the arrows of $Q$.
It extends to an automorphism $s$ of $\bbZ Q$ by
\[
 s(i,v_j)=(i,v_{-j}){\qquad(i\in\bbZ,\ j\in\bbZ/(2m)\bbZ).}
\]
For each arrow $u\to v$ of $\bbZ Q$, its image under
$s$ is the unique arrow $s(u)\to s(v)$.
Then
\[
 s^2=1,\qquad s\ta=\ta s,\qquad
 s\ro s^{-1}=\ta\ro^{-1}.
\]
Recall that $x=\ro$ and $y=\ta\ro^{-1}$.
Thus
\[
 sxs^{-1}=y,\qquad sys^{-1}=x.
\]
The group of automorphisms of $\bbZ Q$ commuting
with $\ta$ is therefore
\[
 L(m,m)\rtimes\langle s\rangle.
\]
Put $r=xy^{-1}$. Then $r$ has order $m$, and
\[
 srs^{-1}=r^{-1}.
\]
There is also a direct geometric description of its finite part.
\begingroup
Write $(i,v)$ for a vertex of the repetition quiver $\bbZ Q$, and
define its {\emph{height}} by
\[
 h(i,v)=
 \begin{cases}
 2i,&v\text{ is a source of }Q,\\
 2i+1,&v\text{ is a sink of }Q.
 \end{cases}
\]
Every arrow raises height by one. For an automorphism $f$ of this
translation quiver, the difference $h(fw)-h(w)$ {for a vertex $w\in(\bbZ Q)_0$} has the same value
at both ends of every arrow. Connectedness makes this difference
independent of the vertex $w$. {Write $\Aut(\bbZ Q,\ta)$ for the group of quiver automorphisms
commuting with $\ta$.} It therefore defines a homomorphism
\[
 c:\Aut(\bbZ Q,\ta)\longrightarrow\bbZ,
 \qquad c(f)=h(fw)-h(w).
\]
The union of the height-$0$ and height-$1$ vertices is precisely the
alternating slice $Q$. Hence $c(f)=0$ implies that $f$ preserves this
slice. Its restriction is an automorphism of $Q$, and commutation
with $\ta$ determines $f$ on the whole repetition. Conversely, each
automorphism of $Q$ extends in this way. Thus
\[
 \Ker c\mathrel{{\cong}}\Aut(Q)=D_{2m}.
\]
Since $\bbZ$ has no nontrivial finite subgroup, every finite subgroup
of $\Aut(\bbZ Q,\ta)$ lies in this kernel. Its separate
component-shift coordinate is also zero.
\endgroup

For the Kronecker quiver, \cite[Theorem~4.3 and Remark~4.4]{MY} give
\begin{equation}\label{eq:kronecker}
 \DPic_\k(H)\mathrel{{\cong}}\bbZ ^2\times\PGL_2(\k),
 \qquad \Pic(H)=\Out^0(H)=\PGL_2(\k).
\end{equation}
The quotient is torsion-free. Finally, for the other models the inverse
image of the slice automorphism group in $\DPic(H)$ is exactly
$\Pic(H)$: the kernel is $\Out^0(H)$ and the slice stabilizer is
$\Out^0(H)\rtimes\Aut(Q)$, by \cite[Proposition~1.7 and
Theorem~0.1]{MY}. This completes the proof.
\end{proof}

\begin{rmk}
The finite subgroup in \eqref{eq:lattice} must be retained. For example,
type $\Aaff42$ has a hereditary orientation with a rotation of order
two. The above calculation uses the general restriction theorem of
\cite{MY} and the repetition quiver directly; no torsion-free formula
for the entire affine vertex group is being assumed.
\end{rmk}

\begin{prp}[Hereditary reduction]\label{prop:reduction}
Let $B$ be piecewise hereditary of type $\Aaff pq$, and let $H$ be
the {model chosen in \rmkref{rmk:hereditary-models}}. There is a tilting spectroid for $B$, equivalent
to $H$, which is stable up to isomorphisms under every automorphism of
$B$. In particular, {for every $\ph\in\Aut(B)$ and every positive integer $n$,}
\[
 T_\ph^n(B)\dereq T_\th^n(H)
\]
\begingroup
for some $\th\in\Aut(H)$. More precisely, for a chosen standard
equivalence $F:\Db(\mod B)\to\Db(\mod H)$ with quasi-inverse
$F^-$, the induced derived action of $\th$ is naturally isomorphic
to $F\Ph F^-$.
\endgroup
\end{prp}
\begin{proof}
\begingroup
Choose a standard equivalence
$F:\Db(\mod B)\to\Db(\mod H)$ and a quasi-inverse $F^-$.
Write $F\Pic(B)F^{-1}$ for the subgroup of $\DPic_\k(H)$ obtained
by conjugating classes, as in Section~1.
Conjugation identifies $\Out^0(B)$ with $\Out^0(H)$ by
\cite[Theorem~1]{Yek}. 
{
Put
\[
 K:=q_H\bigl(F\Pic(B)F^{-1}\bigr).
\]
Since $B$ is basic, we have
$\Pic(B)=\Out_\k(B)$.
The algebraic group $\Out_\k(B)$ has only finitely many
connected components, and each of them is a coset of
$\Out^0(B)$.
Hence $\Pic(B)/\Out^0(B)$ is finite.
Since conjugation by $F$ identifies $\Out^0(B)$ with
$\Out^0(H)$, it induces an isomorphism
\[
 \Pic(B)/\Out^0(B)\cong K.
\]
Therefore $K$ is finite.
}
\lemref{lem:finite} gives $K\subseteq q_H(\Pic(H))$.
Since ${\Ker} q_H=\Out^0(H)\subseteq\Pic(H)$, taking inverse images
gives the required inclusion:
\[
 F\Pic(B)F^{-1}
 \subseteq q_H^{-1}(K)
 \subseteq q_H^{-1}\bigl(q_H(\Pic(H))\bigr)=\Pic(H).
\]
Thus every transported automorphism preserves the projective
spectroid of $H$ up to isomorphism. Let $\calE$ be the full
subcategory of $\Db(\mod B)$ on the objects $F^-(P_x)$, {where $P_x=H(\blank,x)$ for each object $x$ of $H$.}
This is a tilting spectroid equivalent to $H$.
For each $\ph\in\Aut(B)$, the inclusion just proved gives an
automorphism $\th\in\Aut(H)$, with induced derived functor $\Theta$,
and a natural isomorphism
\[
 F\Ph F^-\mathrel{{\cong}}\Theta,
 \qquad\text{hence}\qquad \Ph F^-\mathrel{{\cong}} F^-\Theta.
\]
Since $\Theta$ permutes the projectives up to isomorphism, this proves
that $\calE$ is $\Ph$-stable {up to isomorphism}.
The same $\calE$ works for every
$\ph$, but no coherent action of the whole automorphism group is
needed: \prpref{prop:tilting} is applied separately to the cyclic
action generated by the chosen $\ph$.
\endgroup
\end{proof}

\begin{proof}[Proof of \thmref{thm:main} for twisted extensions]
\begingroup
Recall that
\[
 \La=T_\ph^n(B),
 \qquad
 \La'=T_\ps^{n'}(B').
\]
Here $B$ and $B'$ are piecewise hereditary algebras
of affine type, $\ph\in\Aut(B)$,
$\ps\in\Aut(B')$, and $n,n'\geq1$.

\medskip
\noindent
{(\ref{item:main-derived})} $\Rightarrow$ {(\ref{item:main-stable})}.
As in the proof of \thmref{thm:hereditary}, this follows
from Rickard's theorem \cite{Ric89}: a derived equivalence
between selfinjective algebras induces a triangle equivalence
between their stable module categories.

\medskip
\noindent
{(\ref{item:main-stable})} $\Rightarrow$ {(\ref{item:main-type})}.
Suppose that $\La$ and $\La'$ are stably equivalent.
The tree cases and their separation from type
$\widetilde A$ were handled in Section~\ref{sec:tree}.
We may therefore assume that both $B$ and $B'$ have
 type $\widetilde A$.
Apply \prpref{prop:reduction} to $(B,\ph)$.
There are a chosen hereditary model $H$ and
an automorphism $\ph_0\in\Aut(H)$ such that
\[
 \La=T_\ph^n(B)\dereq T_{\ph_0}^n(H)=:E.
\]
Apply the same proposition separately to $(B',\ps)$.
There are a chosen hereditary model $H'$ and
an automorphism $\ps_0\in\Aut(H')$ such that
\[
 \La'=T_\ps^{n'}(B')\dereq T_{\ps_0}^{n'}(H')=:E'.
\]
By {(\ref{item:main-derived})} $\Rightarrow$ {(\ref{item:main-stable})}, these reduction equivalences
also induce stable equivalences.
Hence $E$ and $E'$ are stably equivalent.
By \prpref{prop:recover}, $H$ and $H'$ have the same
hereditary type $\De$, and $n=n'$.
Since we use the fixed model for each type, we have
$H=H'=H_\De$.
The implication {(\ref{item:hereditary-stable})} $\Rightarrow$ {(\ref{item:hereditary-type})} of
\thmref{thm:hereditary} now shows that the derived
autoequivalences induced by $\ph_0$ and $\ps_0$ are
conjugate in $\DPic_\k(H_\De)$.
By \prpref{prop:reduction}, their conjugacy classes
are precisely those obtained by transporting the
original actions of $\ph$ and $\ps$.
Thus all three components of the two triples agree,
and hence $\typ(\La)=\typ(\La')$.

\medskip
\noindent
{(\ref{item:main-type})} $\Rightarrow$ {(\ref{item:main-derived})}.
Suppose that $\typ(\La)=\typ(\La')$.
Thus $B$ and $B'$ have the same hereditary type $\De$,
$n=n'$, and the third components of their triples agree.
The tree-type case was proved in Section~\ref{sec:tree}.
We may therefore assume that $\De$ has type
$\widetilde A$.

Apply \prpref{prop:reduction} to $(B,\ph)$, using
the chosen hereditary algebra $H_\De$.
There is an automorphism $\ph_0\in\Aut(H_\De)$ such that
\[
 \La=T_\ph^n(B)\dereq T_{\ph_0}^n(H_\De)=:E.
\]
Apply the same proposition separately to $(B',\ps)$,
using the same $H_\De$.
There is an automorphism $\ps_0\in\Aut(H_\De)$ such that
\[
 \La'=T_\ps^{n'}(B')\dereq T_{\ps_0}^{n'}(H_\De)=:E'.
\]
By \prpref{prop:reduction}, the derived autoequivalences
induced by $\ph_0$ and $\ps_0$ represent the third
components of the original triples.
Since these components agree, the two induced
autoequivalences are conjugate in $\DPic_\k(H_\De)$.
By the implication {(\ref{item:hereditary-type})} $\Rightarrow$ {(\ref{item:hereditary-derived})} of
\thmref{thm:hereditary}, we have $E\dereq E'$.
Therefore {(\ref{item:main-derived})} follows because
\[
 \La\dereq E\dereq E'\dereq\La'.
\]

We now prove that the triple in \dfnref{def:type}
is independent of the chosen presentation.
Suppose that an algebra $A$ has two presentations
as twisted multifold extensions:
\[
 A\cong T_\ph^n(B),
 \qquad
 A\cong T_\ps^{n'}(B').
\]
Then $T_\ph^n(B)$ and $T_\ps^{n'}(B')$ are isomorphic,
and hence stably equivalent.
Applying {(\ref{item:main-stable})} $\Rightarrow$ {(\ref{item:main-type})}, which we have just proved,
we obtain
\[
 (\De,n,[\ga_{B,\ph}])
 =
 (\De',n',[\ga_{B',\ps}]),
\]
where $\De$ and $\De'$ are the hereditary types of
$B$ and $B'$, respectively.
Thus the two presentations give the same triple,
as required.
\endgroup
\end{proof}

\section{Explicit parameters in type \texorpdfstring{$\widetilde A$}{A}}
\label{sec:parameters}

\begingroup
The preceding sections establish the type as a complete invariant for derived and stable equivalence of twisted multifold extensions. Our next aim is to determine explicitly all types that actually occur.
In this section we treat the affine types $\tilde{A}$; the affine types $\tilde{D}$ and $\tilde{E}$ are
treated in {Section~\ref{sec:treetable}.}
Since the folding number can be any positive integer, it remains to determine the possible third components of the type.
For each affine type $\De$, we start with the set
\[
 {\calS_\De}
 :=\bigl\{\ga_{B,\ph}{\mid}
 B\text{ is piecewise hereditary of type }\De,
 \ \ph\in\Aut(B)\bigr\}
 \subseteq\DPic_\k(H_\De),
\]
where $\ga_{B,\ph}$ is defined using the standard equivalences chosen in Section 1.
Our task is to determine the set of conjugacy classes represented by these elements.
By \prpref{prop:reduction},
\[
 \bigl\{[\ga]\mid\ga\in{\calS_\De}\bigr\}
 =
 \bigl\{[x]\mid x\in\Pic_\k(H_\De)\bigr\},
\]
where all conjugacy classes are taken in $\DPic_\k(H_\De)$.
Indeed, the hereditary reduction allows each transported twist to be represented by an ordinary automorphism of $H_\De$.
Conversely, since $H_\De$ is basic, every element of $\Pic_\k(H_\De)$ comes from such an automorphism and is realized by taking $B=H_\De$.
We determine this set in three steps.
First, we describe the elements of $\Pic_\k(H_\De)=\Out_\k(H_\De)$ by explicit parameters.
Next, we determine their conjugacy classes within $\Pic_\k(H_\De)$.
Finally, we determine whether distinct classes obtained in this way become identified under conjugation in the larger group $\DPic_\k(H_\De)$.
This last step is necessary because the third component of the type is a conjugacy class in the derived Picard group, not merely in the Picard group.

For twists that rescale arrows, we define the scalar coordinate
$\la\in\kk$ by fixing a direction around the underlying cycle.
Take the product of the arrow scalars along this direction, using
the inverse scalar when an arrow points against it.
This product is called the \emph{holonomy}.
Reversing the chosen direction replaces $\la$ by $\la^{-1}$.
We keep the direction fixed when comparing parameters.
{The cases $p>q=1$ (the \emph{bypass case}) and
$p=q=1$ (the \emph{Kronecker case}) can also have twists
that mix paths or arrows; their parameters are described
separately below.}
\par\endgroup

\Needspace{18\baselineskip}
\begin{thm}\label{thm:table}
{For fixed $\De=\Aaff pq$ and a positive integer $n$,} the derived equivalence
classes, equivalently the stable equivalence classes, are parametrized by
\begin{center}
\renewcommand{\arraystretch}{1.35}
\begin{tabular}{@{}ll@{}}
\toprule
Type & Parameter set \\
\midrule
$p>q>1$, $d=\gcd(p,q)$ & $\kk\times{\bbZ/d\bbZ}$ \\
$p=q=m>1$ &
$\displaystyle(\kk\times{\bbZ/m\bbZ})/\!\sim\ \sqcup\{\mathsf R\}$ \\
$p>q=1$ & $\kk\sqcup\{\mathsf J\}$ \\
$p=q=1$ & $\displaystyle\kk/(\la\sim\la^{-1})\sqcup\{\mathsf J\}$ \\
\bottomrule
\end{tabular}
\end{center}
\begingroup
Here ${\bbZ/d\bbZ}$ and ${\bbZ/m\bbZ}$ record the finite rotation exponent.
{In the second row, for all $(\la,a)\in\kk\times\bbZ/m\bbZ$, there is one simultaneous identification:}
\[
 (\la,a)\sim(\la^{-1},-a).
\]
Thus the two coordinates cannot be inverted independently.
The symbol $\mathsf R$ denotes the single reflection class.
The symbol $\mathsf J$ denotes the nonidentity unipotent class,
whose representatives are given below. The identity is represented
by scalar $1$ and, when there is a rotation coordinate, exponent $0$.
The symbol $\sqcup$ means that the displayed extra class is added as
a separate point.
\par\endgroup
\end{thm}

We prove the four rows of \thmref{thm:table} in the
following four subsections, respectively.
For the first row, we first compute the group in
\lemref{lem:meshgroup}.

\subsection{The mesh group when \texorpdfstring{$p,q>1$}{p,q>1}}

Let $H=\k Q$ be the chosen hereditary algebra.
As in \prpref{prop:MY}, let $\calB$ be the full
subcategory of $\Db(\mod H)$ obtained by choosing
one representative of each isomorphism class of
indecomposable objects in the transjective component
containing the projective $H$-modules.

\begin{lem}\label{lem:meshgroup}
If $p,q>1$, the full component automorphism group is
\begin{equation}\label{eq:meshgroup}
 \Out_\k(\calB)\mathrel{{\cong}}
 \begin{cases}
 \kk\times L(p,q),&p>q,\\
 (\kk\times L(m,m))\rtimes\langle s\rangle,&p=q=m,
 \end{cases}
\end{equation}
{where, in the second line, $s$ acts by $s\la s^{-1}=\la^{-1}$ for every $\la\in\kk$, $sxs^{-1}=y$,
and $sys^{-1}=x$. Here $\kk$ is identified with the first factor, and $x,y$ are the generators of $L(m,m)$ in
\eqref{eq:lattice}.}
\end{lem}
\begin{proof}
{The group $\Aut(\bbZ Q,\ta)$ was computed in \lemref{lem:finite}. We describe the kernel of the
homomorphism from $\Out_\k(\calB)$ to this group induced by the action on objects, and a splitting of this
homomorphism. Use the grid with vertices $(a,b)\in\bbZ^2$ and arrows $(a,b)\to(a+1,b)$ and $(a,b)\to(a,b+1)$
for all $a,b\in\bbZ$. Its quotient by the translation $(a,b)\mapsto(a+p,b-q)$ is a model of $\bbZ Q$; the
elementary squares give its meshes.} For $p,q>1$ there is no path of length greater than one with the
same endpoints as an arrow. Indeed a second path would change the two
nonnegative step counts by a nonzero multiple of $(p,-q)$, starting
from $(1,0)$ or $(0,1)$, which is impossible with both $p,q>1$.
Consequently an object-fixing mesh automorphism only rescales arrows.

Preserving each mesh relation says that the products of the rescaling
factors along its two sides are equal. {We call such scalar assignments
\emph{multiplicative cocycles}. Filling the elementary squares gives a square complex
whose universal cover is the grid and whose periodic quotient is a cylinder. Modulo
rescaling at vertices,} the remaining datum is its single multiplicative
holonomy, an element of $\kk$. Equivalently, one can propagate a
vertex rescaling along a spanning tree and use the mesh squares to
leave exactly the product around the cycle.

There is a splitting on the vertex group: send every arrow to the
arrow with the transported endpoints. There are no parallel arrows
in this case. {The \emph{polarization} pairs each arrow $u\to v$ with the corresponding arrow $\ta
v\to u$ in the mesh at $v$. These quiver actions preserve this pairing and the mesh sums} and satisfy the same
group relations as their
vertex permutations. Translations preserve the holonomy, whereas a
reversal inverts it. This proves \eqref{eq:meshgroup}.
\end{proof}

For the periodic unbalanced orientation, every ordinary twist modulo
inner automorphisms is $\la r^a$, {with $\la\in\kk$ and $a\in\bbZ/d\bbZ$. Here $\la$ denotes the
scalar-group element with holonomy $\la$.}
The group in the first line of \eqref{eq:meshgroup} is abelian.
The extra shift factor in \eqref{eq:mesh} is central. Thus two such
twists are derived conjugate if and only if both $\la$ and $a$
are equal. This proves the first row of \thmref{thm:table}.

\subsection{Balanced type and simultaneous conjugacy}

Let $p=q=m>1$. Ordinary twists on the alternating model {of
\rmkref{rmk:hereditary-models}} have the
forms $\la r^a$ and $\la r^as$, {with $\la\in\kk$ and $a\in\bbZ/m\bbZ$.} Here $r=xy^{-1}$ has order
$m$, and
\begin{equation}\label{eq:reflectionmerge}
 xsx^{-1}=rs.
\end{equation}
\begingroup
First consider twists of the form $\la r^a$.
Conjugation by the abelian normal subgroup fixes them.
The reflection $s$ reverses both the cycle direction and the rotation:
\[
 s\la s^{-1}=\la^{-1},\qquad
 srs^{-1}=r^{-1},\qquad
 s(\la r^a)s^{-1}=\la^{-1}r^{-a}.
\]
Hence the same conjugation changes both coordinates.
\begin{center}
\begin{tikzpicture}[>=stealth, every node/.style={}]
\node (left) at (0,0) {$\la r^a$};
\node (right) at (5,0) {$\la^{-1}r^{-a}$};
\draw[<->] (left) -- node[above] {conjugation by $s$} (right);
\node at (0,-0.7) {$(\la,a)$};
\node at (5,-0.7) {$(\la^{-1},-a)$};
\end{tikzpicture}
\end{center}
These are exactly the identifications in the second row of the table.
\par\endgroup

\begingroup
Next consider twists of the form $\la r^as$. 
Since $s$ reverses the direction around the cycle, we have
$sts^{-1}=t^{-1}$ for every $t\in\k^\times$.
Hence
\[
 t(\la r^a s)t^{-1}
 =t\la r^a(st^{-1})
 =t\la r^a t s
 =t^2\la r^a s.
\]
Choose {$t\in\k^\times$} with $t^2=\la^{-1}$. Such a nonzero $t$ exists because
$\k$ is algebraically closed, also in characteristic two.
This makes the scalar equal to $1$.
Relation~\eqref{eq:reflectionmerge} then changes $a$ by one,
so every exponent can be reduced to zero.
All reflection twists therefore belong to the one class $\mathsf R$.
This proves the second row.
\par\endgroup

\begin{exm}
For $m=3$ over $\bbC$, we show directly that
$2r$ and $2r^2$ are not conjugate in $\DPic_\k(H)$.
Put $A=\k^\times\times L(3,3)$.
By \prpref{prop:MY} and \lemref{lem:meshgroup},
we have
\[
 \DPic_\k(H)\cong
 (A\rtimes\langle s\rangle)\times\bbZ.
\]
The shift factor $\bbZ$ is central, so it does not
affect conjugation.
Every element of $A\rtimes\langle s\rangle$ has
the form $h$ or $hs$ for some $h\in A$.
Since $A$ is abelian, {for all $h\in A$, $\la\in\kk$, and $a\in\bbZ/3\bbZ$,} we have
\[
 h(\la r^a)h^{-1}=\la r^a
\]
and
\[
 \begin{aligned}
 (hs)(\la r^a)(hs)^{-1}
 &=h\bigl(s(\la r^a)s^{-1}\bigr)h^{-1}\\
 &=h(\la^{-1}r^{-a})h^{-1}\\
 &=\la^{-1}r^{-a}.
 \end{aligned}
\]
Thus the conjugacy classes of $2r$ and $2r^2$
in $\DPic_\k(H)$ are, respectively,
\[
 \{2r,(1/2)r^2\},\qquad
 \{2r^2,(1/2)r\}.
\]
Here we use $r^3=1$.
Since $r$ has order $3$ and $2\neq1/2$ in $\bbC$,
these two conjugacy classes are distinct.
In terms of the parameters, they are
\[
 \{(2,1),(1/2,2)\},\qquad
 \{(2,2),(1/2,1)\}.
\]
However, the graph parts $r$ and $r^2$ are conjugate
by $s$, and the scalar parts of $2r$ and $2r^2$
are both $2$.
Thus recording the graph conjugacy class and the
scalar modulo inversion separately does not
distinguish these two conjugacy classes in
$\DPic_\k(H)$.
The reason is that conjugation by $s$ inverts
both parts simultaneously:
\[
 s(\la r^a)s^{-1}=\la^{-1}r^{-a}.
\]
Therefore these two separate data do not form
a complete invariant.
\end{exm}

The graph component corresponds to the third component of
\cite[Definition~8.2]{Asa}, and is given by a conjugacy class.
Its order alone is insufficient. In the balanced case, the graph
component and the scalar are transformed by the same conjugation,
as the example above shows.

\subsection{The bypass case \texorpdfstring{$p>q=1$}{p>q=1}}

\begingroup
Let $H=\k Q$ be the path category of the quiver $Q$
with objects $0,1,\ldots,p$, arrows
$\al_i:i-1\to i$ {for every $i\in\{1,\ldots,p\}$,}
and an additional arrow $b:0\to p$.
Let $P$ be the path from $0$ to $p$ consisting of
$\al_1,\ldots,\al_p$.
Thus $H(0,p)$ has basis $b,P$.
We first determine the conjugacy classes in
$\Pic_\k(H)$, and then show that distinct such classes
remain distinct in $\DPic_\k(H)$.

\medskip
\setcounter{proofstep}{0}
\noindent\textbf{\proofstep{step:bypass-aut} Automorphisms of the category and their conjugacy classes.}
{Since $H(i,j)\neq0$ if and only if $i\leq j$ for all $i,j\in\{0,\ldots,p\}$,}
every automorphism of $H$ fixes every object.
First consider any $\k$-linear endomorphism $\ph$ of $H$ that fixes every object, without assuming that $\ph$ is invertible.
{For each $i\in\{1,\ldots,p\}$, since $H(i-1,i)=\k\al_i$, there is a scalar $c_i\in\k$ such that}
\[
 \ph(\al_i)=c_i\al_i.
\]
Similarly, since $b$ and $P$ form a basis of $H(0,p)$, there are scalars $d,e\in\k$ such that
\[
 \ph(b)=db+eP.
\]
Since $\ph$ preserves composition, the first formula also gives
\[
 \ph(P)=(c_1\cdots c_p)P.
\]
Now suppose that $\ph$ is an automorphism.
Its restriction to each one-dimensional space $H(i-1,i)$ is invertible, so $c_i\neq0$ {for every $i\in\{1,\ldots,p\}$.}
Its restriction to $H(0,p)$ is also invertible.
If $d=0$, then both $\ph(b)$ and $\ph(P)$ would belong to the one-dimensional subspace $\k P$, contradicting this invertibility.
Hence $d\neq0$ as well.
Thus, for an automorphism, the coefficients satisfy $c_i,d\in\k^\times$ and $e\in\k$.

We normalize $\ph$ up to natural isomorphism.
{For each $i\in\{0,\ldots,p\}$, choose $t_i\in\k^\times$. Define an automorphism $\ps$ of $H$ by}
\[
 \ps(i)=i,\qquad
 \ps(f)=\frac{t_j}{t_i}\ph(f).
\]
{These assignments are made for all $i,j\in\{0,\ldots,p\}$ and $f\in H(i,j)$.}
{For each $i\in\{0,\ldots,p\}$, the morphism}
\[
 \eta_i=t_i1_i:\ph(i)\longrightarrow\ps(i)
\]
{is a component of a natural isomorphism $\eta:\ph\Longrightarrow\ps$.}
{Indeed, for all $i,j\in\{0,\ldots,p\}$ and every $f\in H(i,j)$, we have}
\[
 \ps(f)\eta_i
 =\frac{t_j}{t_i}\ph(f)t_i1_i
 =t_j1_j\ph(f)
 =\eta_j\ph(f),
\]
which is precisely the naturality condition.

For each $a\in\k^\times$ and $v\in\k$, let $\ph_{a,v}$ be the automorphism fixing all objects and given on arrows by
\[
 \begin{aligned}
 \ph_{a,v}(\al_1)&=a\al_1,\\
 \ph_{a,v}(\al_i)&=\al_i\quad(2\leq i\leq p),\\
 \ph_{a,v}(b)&=b+vP.
 \end{aligned}
\]
In particular, $\ph_{a,v}(P)=aP$.
We now choose the scalars $t_i$ so that $\ps$ has one of these normal forms.
To ensure that $\ps(\al_i)=\al_i$ for $2\leq i\leq p$ and that the coefficient of $b$ in $\ps(b)$ is $1$, we require
\[
 \frac{t_i}{t_{i-1}}=c_i^{-1}
 \quad(2\leq i\leq p),
 \qquad
 \frac{t_p}{t_0}=d^{-1}.
\]
These conditions can be satisfied by setting $t_0=1$, $t_p=d^{-1}$,
and defining $t_{i-1}=c_i t_i$ successively
for $i=p,p-1,\ldots,2$.
With this choice, $\ps=\ph_{a,v}$, where
\[
 a=\frac{c_1\cdots c_p}{d},
 \qquad
 v=\frac{e}{d}.
\]
Here $d$ and $e$ are the coefficients in $\ph(b)=db+eP$.
Thus every automorphism of $H$ is naturally isomorphic
to one of these normalized automorphisms.

The parameters of a normalized automorphism are
uniquely determined by its natural-isomorphism class.
{Take any $(a,v),(a',v')\in\k^\times\times\k$. Suppose that}
\[
 \eta:\ph_{a,v}\Longrightarrow\ph_{a',v'}
\]
is a natural isomorphism.
{Since $H(i,i)=\k1_i$ for every $i\in\{0,\ldots,p\}$, for each such $i$ there is $t_i\in\k^\times$ with $\eta_i=t_i1_i$.}
Naturality for $b$ gives
\[
 t_0(b+v'P)=t_p(b+vP).
\]
Comparing the coefficients of $b$ gives $t_0=t_p$,
and comparing the coefficients of $P$ then gives
$v=v'$.
Naturality for $\al_i$, with $2\leq i\leq p$, gives
$t_{i-1}=t_i$.
Hence
\[
 t_0=t_1=\cdots=t_p.
\]
Finally, naturality for $\al_1$ gives
\[
 a't_0=at_1,
\]
and therefore $a'=a$.
Conversely, equal parameters give the same automorphism.
This proves the uniqueness of the parameters.

We write $(a,v)$ for the element of $\Pic_\k(H)$
represented by $\ph_{a,v}$.
{For all $(a,v),(a',v')\in\k^\times\times\k$, composition and inversion are given by}
\[
 \begin{aligned}
 (a,v)(a',v')&=(aa',v+av'),\\
 (a,v)^{-1}&=(a^{-1},-a^{-1}v).
 \end{aligned}
\]
This is also the multiplication law for the affine
transformations {of $\k$} given by $z\mapsto az+v$.
Their group is denoted by $\Aff_1(\k)$.
Thus
\[
 \Pic_\k(H)\cong\Aff_1(\k).
\]
This description agrees with
\cite[Theorem~4.5(2)]{MY}.

{For all $(a,v),(c,w)\in\Aff_1(\k)$, we obtain}
\begin{equation}\label{eq:affconj}
 (c,w)(a,v)(c,w)^{-1}
 =(a,cv+(1-a)w).
\end{equation}
{The following formulas hold for all $a\in\k^\times$ and $v\in\k$, subject to the indicated
conditions, and give a representative of each conjugacy class:}
\begin{equation}\label{eq:affconj-representatives}
 \begin{aligned}
 &(1,-v/(1-a))(a,v)(1,v/(1-a))=(a,0)
 &&(a\neq1),\\
 &(1,0)=\id,\\
 &(v^{-1},0)(1,v)(v,0)=(1,1)
 &&(v\neq0).
 \end{aligned}
\end{equation}
Conjugation preserves $a$, and the identity cannot be
conjugate to a nonidentity element.
Therefore a complete set of representatives of the
conjugacy classes in $\Pic_\k(H)$ is
\[
 (a,0)\quad(a\in\k^\times),\qquad (1,1).
\]
The class represented by $(1,1)$ is denoted by $\mathsf J$.
Its representative fixes every $\al_i$ and sends $b$ to $b+P$.

\medskip
\noindent\textbf{\proofstep{step:bypass-derived} Conjugation in the derived Picard group.}
It remains to show that conjugation in $\DPic_\k(H)$
gives no further identifications among the conjugacy
classes in $\Pic_\k(H)$.
Our strategy is to use the algebraic group structure
to control this conjugation action.
More precisely, we will show that its restriction
to the translation subgroup
$U=\{(1,u)\mid u\in\k\}$ is multiplication
by a nonzero scalar.
{The conjugation relation between $G=\Pic_\k(H)$ and $U$} will then
force the first coordinate $a$ to be preserved.
Together with \eqref{eq:affconj-representatives},
this will imply that each conjugacy class is preserved.

For this purpose, we use \cite[Theorem~1]{Yek}, which
states that conjugation by an element of $\DPic_\k(H)$
restricts to an automorphism of the algebraic group
$\Out^0(H)$.
To apply this theorem to the whole group
$G=\Pic_\k(H)$, we first show that
$\Pic_\k(H)=\Out^0(H)$.

Since $H$ is basic, we may make the identification
$\Pic_\k(H)=\Out_\k(H)$; see
Appendix~\ref{app:standard-facts}.
The normalization in {(\ref{step:bypass-aut})} gives an isomorphism
\[
 \Out_\k(H)\cong\Aff_1(\k)
\]
also as algebraic groups.
Indeed, the normalized parameters
\[
 a=\frac{c_1\cdots c_p}{d},
 \qquad v=\frac{e}{d}
\]
are regular functions on the group of object-fixing
automorphisms, since $d\neq0$.
They are constant on natural-isomorphism classes.
The inverse parametrization is induced by the regular
family of normalized automorphisms
$(a,v)\mapsto\ph_{a,v}$.
Thus the underlying variety of $\Out_\k(H)$ is
$\k^\times\times\k$, which is connected.
Consequently,
\[
 \Out^0(H)=\Out_\k(H)=\Pic_\k(H).
\]

Put $G=\Aff_1(\k)$, identified with this subgroup of
$\DPic_\k(H)$, and put
\[
 U=\{(1,u)\mid u\in\k\}.
\]
The group $U$ consists of the translations $z\mapsto z+u$.

First we show that $U$ is the commutator subgroup of $G$.
The map $G\to\k^\times$, $(a,v)\mapsto a$, is a group
homomorphism with kernel $U$, where
$\k^\times$ is abelian. Hence every commutator belongs to $U$.
Conversely, choose $c_0\in\k^\times$ with $c_0\neq1$;
such an element exists because $\k$ is algebraically closed.
For every $w\in\k$, we have
\[
 (c_0,0)(1,w)(c_0,0)^{-1}(1,w)^{-1}
 =(1,(c_0-1)w).
\]
For every $u\in\k$, taking $w=u/(c_0-1)$ expresses
$(1,u)$ as a commutator.
This proves the assertion.
In particular, every group automorphism of $G$ preserves $U$.

Now take any $F\in\DPic_\k(H)$.
By \cite[Theorem~1]{Yek}, conjugation by $F$ induces
an automorphism of algebraic groups
\[
 \th:G\to G,\qquad g\longmapsto FgF^{-1},
\]
where we use $G\cong\Out^0(H)$.
We use the algebraic nature of $\th$ only in the
following step.
Since $\th(U)=U$, there is a bijection $f:\k\to\k$
such that
\[
 \th(1,u)=(1,f(u))\quad\text{for every }u\in\k.
\]
The coordinate $u$ identifies $U$ with the affine line.
Here we use the fact that $\th$ is an automorphism
of algebraic groups, not merely of abstract groups.
Since $U\cong\mathbb A^1_\k$ and both $\th$ and
$\th^{-1}$ are morphisms of algebraic varieties,
the induced map $f$ and its inverse are polynomial maps.
Their composite is the identity polynomial, so
\[
 \deg(f)\deg(f^{-1})=1.
\]
Thus $f$ has degree one.
Since $\th$ preserves the identity, $f(0)=0$.
{There is therefore a scalar $c\in\k^\times$ such that $f(u)=cu$ for every $u\in\k$.}
Consequently, %
\[
 \th(1,u)=(1,cu)\quad\text{for every }u\in\k.
\]
This argument is valid in every characteristic.

Let $g=(a,v)\in G$, and write $\th(g)=(a',v')$.
For every $u\in\k$, direct calculation gives
\[
 g(1,u)g^{-1}=(1,au).
\]
Applying $\th$ to both sides, we obtain
\[
 \begin{aligned}
 (1,cau)
 &=\th\bigl(g(1,u)g^{-1}\bigr)\\
 &=\th(g)(1,cu)\th(g)^{-1}\\
 &=(1,a'cu).
 \end{aligned}
\]
Taking $u=1$ and using $c\neq0$, we obtain $a=a'$.
Thus $\th$ preserves the first coordinate.
We now apply \eqref{eq:affconj-representatives}
to the images of the representatives found in {(\ref{step:bypass-aut})}.
\begin{itemize}
\item
{For $a\in\k^\times\setminus\{1\}$, write $\th(a,0)=(a,v')$.}
The first formula shows that $\th(a,0)$ is conjugate
in $G$ to $(a,0)$.
\item
Since $\th$ is a group automorphism,
$\th(1,0)=(1,0)$.
\item
We have $\th(1,1)=(1,c)$ with $c\neq0$.
The third formula, applied with $v=c$, shows that
$\th(1,1)$ is conjugate in $G$ to $(1,1)$.
\end{itemize}
Thus, in each of the three cases, the image under $\th$
of the chosen representative is conjugate in $G$
to that representative.
Hence $\th$ sends each conjugacy class in
$G=\Pic_\k(H)$ to itself.
Since $F\in\DPic_\k(H)$ was arbitrary and
$\th(g)=FgF^{-1}$, conjugation in $\DPic_\k(H)$
gives no further identifications among these classes.
Therefore the representatives found in {(\ref{step:bypass-aut})} give exactly
the conjugacy classes in $\DPic_\k(H)$ that meet
$\Pic_\k(H)$.
This proves the third row of \thmref{thm:table}.
\par\endgroup

\subsection{The Kronecker case}

Under the decomposition in \eqref{eq:kronecker},
$\Pic_\k(H)$ corresponds to
$\{0\}\times\PGL_2(\k)$.
For any $z\in\bbZ^2$ and $h,x\in\PGL_2(\k)$, we have
\[
 (z,h)(0,x)(z,h)^{-1}=(0,hxh^{-1}).
\]
Thus conjugation in $\DPic_\k(H)$ gives no further
identifications among elements of $\Pic_\k(H)$ beyond
conjugation in $\Pic_\k(H)$ itself.
It therefore suffices to determine the conjugacy classes
in $\PGL_2(\k)$.

{For each $A\in\GL_2(\k)$, write $[A]$ for its image in $\PGL_2(\k)$. For any $A,B\in\GL_2(\k)$, the
elements $[A]$ and $[B]$ are conjugate if and only if}
$B=cPAP^{-1}$ for some $P\in\GL_2(\k)$ and
$c\in\k^\times$.
Since $\k$ is algebraically closed, Jordan normal form
gives the following two cases.

\smallskip
\setcounter{proofstep}{0}
\noindent\textbf{\proofstep{case:Kronecker-diagonal} Diagonalizable matrices.}
Let $\al,\beta\in\k^\times$ be the eigenvalues of $A$.
By conjugation and multiplication by $\beta^{-1}$,
we obtain the representative
\[
 D_\la=\begin{bmatrix}\la&0\\0&1\end{bmatrix},
 \qquad \la=\al/\beta\in\k^\times.
\]
Conjugation preserves the unordered pair of eigenvalues,
and multiplication by a nonzero scalar preserves their ratio.
{Thus, for all $\la,\mu\in\k^\times$, $[D_\la]$ and $[D_\mu]$ are conjugate if and only if}
$\mu=\la$ or $\mu=\la^{-1}$, the latter possibility
corresponding to interchanging the two eigenvalues.
The value $\la=1$ gives the identity class.

\smallskip
\noindent\textbf{\proofstep{case:Kronecker-Jordan} Non-diagonalizable matrices.}
In this case, $A$ is conjugate to a single Jordan block
with a nonzero eigenvalue.
Multiplication by a nonzero scalar followed by conjugation
gives the representative
\[
 J=\begin{bmatrix}1&1\\0&1\end{bmatrix}.
\]
Thus all non-diagonalizable matrices represent a single
conjugacy class in $\PGL_2(\k)$, namely the nontrivial
unipotent class represented by $J$.

\smallskip
Since diagonalizability is preserved by conjugation and
multiplication by a nonzero scalar, the class in {(\ref{case:Kronecker-Jordan})}
is distinct from every class in {(\ref{case:Kronecker-diagonal})}.
The two cases exhaust all invertible $2\times2$ matrices,
so the classes listed above exhaust all conjugacy classes
in $\PGL_2(\k)$.
This classification is valid in arbitrary characteristic
and proves the last row of \thmref{thm:table}.

\section{The finite parameters in the affine tree types}
\label{sec:treetable}
We now list all possible third components of the type
for the affine tree types.
These are conjugacy classes in $\DPic_\k(H_\De)$,
as in \dfnref{def:type}.
By Section~\ref{sec:tree}, they are parametrized by
those conjugacy classes in $\Aut(\De)$ that have
representatives in $\Aut_0(\De)$, via
\[
 \ovl\si\longmapsto
 [\ga_{H_\De,\k\si}].
\]
The table lists representatives $\si$ of these graph
conjugacy classes.
The last column gives the number of corresponding
realizable conjugacy classes in $\DPic_\k(H_\De)$.

\begin{center}
\renewcommand{\arraystretch}{1.25}
\begin{tabular}{@{}lll@{}}
\toprule
\shortstack{Type\\$\De$}
& \shortstack{Representatives of realizable\\graph classes: $\si$}
& \shortstack{Number of\\classes} \\
\midrule
$\widetilde D_4$ & $1,(12),(12)(34),(123),(1234)$ in $S_4$ & $5$ \\
$\widetilde D_r$, even $r\geq6$ & $1,a,ab,s,as$ & $5$ \\
$\widetilde D_r$, odd $r\geq5$ & $1,a,ab$ & $3$ \\
$\widetilde E_6$ & $1,(12),(123)$ in $S_3$ & $3$ \\
$\widetilde E_7$ & $1$, the nontrivial symmetry & $2$ \\
$\widetilde E_8$ & $1$ & $1$ \\
\bottomrule
\end{tabular}
\end{center}

Here, for $\widetilde D_r$ with $r\geq5$, $a$ and $b$ exchange the
two terminal leaves at the two respective
forks, and $s$ interchanges
the forks. {Here the forks are the two vertices of degree three, and the
leaves are the vertices of degree one.} The graph automorphism group has presentation
\[
 \langle a,b,s\mid a^2=b^2=s^2=1,\ ab=ba,\ sas=b\rangle
\]
and conjugacy classes
\[
 \{1\},\quad\{a,b\},\quad\{ab\},\quad
 \{s,abs\},\quad\{as,bs\}.
\]
For even $r$ there is a central vertex, fixed by the whole group.
For odd $r$ the last two classes exchange the endpoints of the central
edge and have no fixed vertex. They therefore do not occur as ordinary
twists on an admissible orientation. Nevertheless the ambient
conjugacy group remains the full graph automorphism group, so $a$ and
$b$ are identified. For $\widetilde D_4$, permutation of the four leaves
gives $S_4$. For $\widetilde E_6$, permutation of the three equal arms
gives $S_3$; the other two exceptional diagrams give the groups shown.
Each representative $\si$ in the table and each positive
integer $n$ give the type
\[
 \bigl(\De,n,[\ga_{H_\De,\k\si}]\bigr).
\]

\begin{rmk}
In the representation-finite classification, the third
component of the type is given by the torsion order.
In the present setting, however, the order of the graph
automorphism does not suffice: its conjugacy class
must be retained.
For example, in type $\widetilde D_4$, the automorphisms
$(12)$ and $(12)(34)$ both have order two but belong
to distinct conjugacy classes in $S_4$.
Both classes are realizable and, for the same positive
integer $n$, give distinct types.
\end{rmk}

\section{Generalized multifold extensions}\label{sec:generalized}

\begingroup
{Let $B$ be a finite-dimensional basic connected $\k$-algebra, and let $g\in\Aut(\widehat B)$ have nonzero integer jump.}
Replace $g$ by $g^{-1}$ if necessary, so that its jump is $n>0$.
This leaves the orbit algebra unchanged because
$\langle g\rangle=\langle g^{-1}\rangle$.
Thus $g$ moves every layer by the positive integer $n$, as in Section~2.
The notation $g$ is reserved for this automorphism of the {repetitive category $\widehat B$};
$\ph$ continues to denote an ordinary automorphism of $B$.
A general $g$ need not have the form $\widehat\ph\nu_B^n$.

To extract an ordinary automorphism, first remove the jump on the
right and then restrict to the zeroth layer. With the layer embeddings
$1^{[i]}:B\longrightarrow B^{[i]}$, put
\[
 g_0:=(1^{[0]})^{-1}g\nu_B^{-n}1^{[0]}\in\Aut(B).
\]
The order in this formula is part of the convention:
\[
 B\xrightarrow{\ 1^{[0]}\ }B^{[0]}
 \xrightarrow{\ \nu_B^{-n}\ }B^{[-n]}
 \xrightarrow{\ g\ }B^{[0]}
 \xrightarrow{\ (1^{[0]})^{-1}\ }B.
\]
For $g=\widehat\ph\nu_B^n$, this gives $g_0=\ph$.
For a general $g$, we do not assume that $g$ and $\nu_B$ commute
on morphisms. 
The definition of $g_0$ follows the convention in
\cite{AKNY}.

\begin{thm}[Asashiba--Kimura--Nakashima--Yoshiwaki]
\label{thm:normalization}
Let $B$ be piecewise hereditary. Choose the positive-jump generator
$g\in\Aut(\widehat B)$ of the cyclic action, with jump $n>0$.
For the automorphism $g_0$ defined above, there is an isomorphism
of $\bbZ$-graded categories
\begin{equation}\label{eq:normalization}
 \widehat B/\langle g\rangle\mathrel{{\cong}} T_{g_0}^{n}(B).
\end{equation}
\end{thm}
\begin{proof}
This is \cite[Corollary~5.3]{AKNY}. Its hypothesis is that $B$
is piecewise hereditary; there is no tree hypothesis in the statement.
The proof uses the absence of nonzero oriented cycles from
\cite[Lemma~5.2]{AKNY}, followed by \cite[Corollary~4.7]{AKNY}.
\end{proof}

The cited result allows any nonzero integer jump. We use its
positive-jump case, so the right-hand side is exactly the twisted
extension defined in Section~2. The automorphism $g_0$ is always
computed after choosing the positive-jump generator.
The automorphism $g_0$ is defined on morphisms as well
as on objects; we retain its full action on $B$,
not merely its permutation of the objects.

\begin{cor}\label{cor:generalized}
For generalized multifold extensions of piecewise hereditary algebras
of the types listed in formula~\eqref{eq:types}, with nonzero integer jump, choose
the positive-jump generator $g$ and let $\Ph_0$ be the standard
autoequivalence induced by $g_0$. The complete derived and stable
equivalence invariant is
\[
 \bigl(\De,n,[\operatorname{cl}(F\Ph_0F^-)]\bigr),\qquad n>0.
\]
{Here $F:\Db(\mod B)\to\Db(\mod H_\De)$ is a standard derived equivalence to the chosen hereditary
model, and $F^-$ is its chosen quasi-inverse.}
The symbol $\operatorname{cl}$ takes the natural-isomorphism class
of the autoequivalence. The brackets then take its conjugacy class
in $\DPic_\k(H_\De)$, as in Section~1.
The explicit third components are those of
\thmref{thm:table} and Section~\ref{sec:treetable}.
\end{cor}
\begin{proof}
Choose the positive-jump generator and apply
\thmref{thm:normalization}. An isomorphism of graded categories
is also an isomorphism after forgetting the grading.
Apply \thmref{thm:main} to the resulting twisted extensions.
This proves the assertion and makes the invariant independent of the
chosen presentation.
Conversely, every representative in the tables is a twisted
extension and hence is a generalized extension. This proves the last
assertion of \thmref{thm:main}.
\end{proof}

The tree-type generalized classification of \cite{AK} is consistent
with this statement. Its tree-specific reduction to the action on
objects should not be used in type $\widetilde A$.
In that type, the morphism action can retain a scalar parameter.
The isomorphism in formula~\eqref{eq:normalization}
reduces the classification to that of $T_{g_0}^n(B)$,
with the full automorphism $g_0$ retained.
In particular, it does not replace $g_0$ by its action
on objects alone.
\par\endgroup

\section{Consequences and examples}\label{sec:examples}

\begin{exm}[The two-arm scalar family]\label{ex:scalar-family}
Let $B$ be the path algebra of two paths from a common source to a
common sink, of lengths $p$ and $q$, with $p>q>1$. {Fix an arrow on the first path. For each
$\la\in\kk$, let $\ph_\la$ fix all vertices and the other arrows and multiply the chosen arrow by $\la$.}
Normalize the cycle holonomy
so that this twist has coordinate $\la$. {Then, for all $\la,\mu\in\kk$ and all positive integers $n,n'$,}
\begin{equation}\label{eq:diagonal}
 T_{\ph_\la}^{n}(B)\dereq T_{\ph_\mu}^{n'}(B)
 \quad\Longleftrightarrow\quad n=n'\text{ and }\la=\mu.
\end{equation}
These twists have trivial finite rotation component. The same
criterion holds for stable equivalence. In particular it applies to
the two-arm $(3,2)$ family. Here $\dim_\k B=14$, so its one-fold
extensions have dimension $28$.
\end{exm}

\begin{cor}\label{cor:HH-family}
Let $\La_\la$ $(\la\in\k^\times)$ be the algebras
of \cite{HH}. For $\la,\mu\in\k^\times$, the following
are equivalent:
\begin{enumerate}
\item\label{item:HH-derived} $\La_\la\dereq\La_\mu$;
\item\label{item:HH-stable} the algebras $\La_\la$ and $\La_\mu$ are stably equivalent;
\item\label{item:HH-scalar} $\la=\mu$.
\end{enumerate}
\end{cor}
\begin{proof}
By the proof of \cite[Proposition~2.1]{HH},
\[
 \La_\la\cong T_{\ph_\la}^1(B),
\]
where $B$ is hereditary of type $\Aaff32$ and
$\ph_\la$ multiplies one arrow on the path of length
$3$ by $\la$ and fixes the other arrows and all vertices.
Choose the direction around the underlying cycle so that
the scalar parameter of this automorphism is $\la$.
Since $3>2>1$ and $\gcd(3,2)=1$, the first row of
\thmref{thm:table} has only this scalar parameter.
Two such parameters represent the same conjugacy class
if and only if they are equal.
The assertion follows from \thmref{thm:main}.
\end{proof}

\begin{exm}[Equal arms]
For the analogous diagonal family with $p=q=m>1$, the rotation
component is again zero, and {for all $\la,\mu\in\kk$ and positive integers $n,n'$,} the criterion becomes
\[
 n=n',\qquad \la=\mu\ \text{or}\ \la=\mu^{-1}.
\]
Once a nontrivial rotation is allowed, this must be replaced by the
joint relation $(\la,a)\sim(\la^{-1},-a)$, as in
\thmref{thm:table}.
\end{exm}

\begin{rmk}[The folding number and the return action]
\begingroup
The folding number and the scalar are different parts of the
invariant:
\[
\begin{aligned}
 n&={\text{number of layers represented in the orbit algebra}},\\
 \la&=\text{scalar coordinate of the twist}.
\end{aligned}
\]
Equation~\eqref{eq:return} recovers the twist itself from the $n$th
power of the Nakayama functor. It does not replace $\la$ by $\la^n$.
Thus the scalar identifications remain those in
\thmref{thm:table}, for each folding number.
\par\endgroup
\end{rmk}

\begin{rmk}[Scope of the invariant]
{The invariant uses a functor on the full subcategory of the stable module category on a transjective component.}
It is therefore finer than the induced permutation of the stable
Auslander--Reiten vertices. The classification also does not assert
that the ordinary Hochschild cohomology algebra is a complete
invariant.
\end{rmk}

\appendix

\section{Equivariant repetitive tilting}\label{app:descent}

In this appendix we prove that an equivariant standard
derived equivalence between two base algebras gives a
derived equivalence between their twisted multifold extensions.
We first recall the covering theorem that we use.
We then construct the required tilting subcategory over
the repetitive category and check its equivariance explicitly.

Let $R$ be a small $\k$-linear category.
A full subcategory $\calE$ of $\Kb(\prj R)$ is
\emph{tilting} if
\[
 \begin{aligned}
 \Kb(\prj R)(X,Y[r])&=0
 \quad\text{for all }X,Y\in\calE,\ {r\in\bbZ\setminus\{0\}},\\
 \operatorname{thick}(\calE)&=\Kb(\prj R).
 \end{aligned}
\]
{Here $\operatorname{thick}(\calE)$ is the smallest full triangulated
subcategory of $\Kb(\prj R)$ that contains $\calE$ and is closed under
isomorphisms in $\Kb(\prj R)$ and under direct summands.}
If a group $G$ acts on $R$, it also acts on
$\Kb(\prj R)$ by the induced functors on complexes.
We call $\calE$ \emph{$G$-stable} if $gX\in\calE$
for every $g\in G$ and every $X\in\calE$.

\begin{thm}[\cite{Asa11}, Theorem~4.11]
\label{thm:Asa11-covering}
Let $R$ and $S$ be small $\k$-linear categories with
actions of a group $G$, and assume that $R$
is $\k$-flat.
If there is a $G$-stable tilting subcategory $\calE$
of $\Kb(\prj R)$ and a $G$-equivariant equivalence
$\calE\to S$, then $R/G$ and $S/G$ are derived equivalent.
\end{thm}

The $\k$-flatness assumption means that all morphism spaces
are flat $\k$-modules; it is automatic over a field.
For small categories, derived equivalence here means
a triangle equivalence between the derived categories
$\calD(\Mod R)$ and $\calD(\Mod S)$ of all right modules.
Equivariance is given by coherent natural isomorphisms,
not necessarily by equalities.
For $G=\bbZ$, it suffices to give one natural isomorphism
between the functors obtained by applying the action of $1\in\bbZ$
before and after the given functor.
Indeed, let $U:\calC\to\calD$ be the given functor, and let
$S$ and $T$ be the automorphisms defining the actions of $1\in\bbZ$
on $\calC$ and $\calD$, respectively.
From a natural isomorphism $\alpha:TU\to US$, define
$\alpha_0=\id_U$ and, for every integer $n\geq0$, set
\[
 \alpha_{n+1}:=(\alpha_n S)\circ(T^n\alpha),
 \qquad
 \alpha_{-n}:=T^{-n}\alpha_n^{-1}S^{-n}.
\]
These natural isomorphisms $\alpha_n:T^nU\to US^n$ satisfy
\begingroup
\begin{equation}\label{eq:Z-equivariance}
 \alpha_{m+n}=(\alpha_m S^n)\circ(T^m\alpha_n)
 \qquad(m,n\in\bbZ).
\end{equation}
\endgroup
{Thus $U$, equipped with $(\alpha_n)_{n\in\bbZ}$, is $\bbZ$-equivariant.}
This construction is the special case of
\cite[Proposition~5.4.15(2)]{AsaBook}
in which the autoequivalences are automorphisms,
the chosen quasi-inverses are their strict inverses,
and the units and counits are identity natural transformations.

\begin{prp}
Let $B$ and $C$ be finite-dimensional basic $\k$-algebras,
regarded as $\k$-linear categories, and let
$U:\Db(\mod C)\to\Db(\mod B)$ be a standard derived
equivalence.
Let $\ph\in\Aut(B)$ and $\ps\in\Aut(C)$ be
$\k$-linear category automorphisms, and denote their induced
derived autoequivalences by $\Ph$ and $\Ps$, respectively.
If $\Ph U\cong U\Ps$ naturally, then
$T_\ph^n(B)$ and $T_\ps^n(C)$ are derived equivalent
for every $n\geq1$.
\end{prp}

\begin{proof}
Fix $n\geq1$, and put
\[
 g_B=\widehat\ph\nu_B^n,
 \qquad
 g_C=\widehat\ps\nu_C^n.
\]
We apply \thmref{thm:Asa11-covering} with
\[
 R=\widehat B,\qquad S=\widehat C,\qquad G=\bbZ,
\]
where $1\in\bbZ$ acts by $g_B$ and $g_C$, respectively.
Both categories are small and $\k$-flat.
Our construction has the following form:
\[
 \begin{aligned}
 C&\xrightarrow{\ F\ }\calE\subseteq\Kb(\prj B),\\
 \widehat C&\xrightarrow{\ \widehat F\ }
 \widehat{\calE}\subseteq\Kb(\prj\widehat B).
 \end{aligned}
\]
The tilting construction in {(\ref{step:repetitive-bases})} and its extension to
repetitive categories in {(\ref{step:repetitive-categories})} follow \cite[proof of Theorem~1.5]{Asa97}.
We recall these constructions in our notation and verify
the equivariance needed to apply
\thmref{thm:Asa11-covering} in {(\ref{step:repetitive-covering})}.

\medskip
\setcounter{proofstep}{0}
\noindent\textbf{\proofstep{step:repetitive-bases} Tilting data over the bases.}

{We use formal additive-hull models of $\prj B$ and $\prj C$.
On projective complexes, $\Ph$ and $\Ps$ are the termwise actions induced
by $\ph$ and $\ps$; their strict inverses are induced by $\ph^{-1}$ and
$\ps^{-1}$, respectively. Negative powers in this construction use these
strict inverses.}
For each $x\in C_0$, put $P_x=C(\blank,x)$ and choose a bounded
complex $T_x$ of projective $B$-modules representing $U(P_x)$.
Since $U$ is a derived equivalence, we have
\[
 \begin{aligned}
 \Kb(\prj B)(T_x,T_y[r])&=0
 \quad(x,y\in C_0,\ {r\in\bbZ\setminus\{0\}}),\\
 \operatorname{thick}\{T_x\mid x\in C_0\}&=\Kb(\prj B).
 \end{aligned}
\]
Choose a natural isomorphism $\et:\Ph U\To U\Ps$.
Using $\Ps P_x\cong P_{\ps x}$ and iterating $\et$
and its inverse gives
\[
 \Ph^jT_x\cong T_{\ps^j x}
 \qquad(x\in C_0,\ j\in\bbZ).
\]
Define $\calE$ to be the full subcategory of $\Kb(\prj B)$
with object set
\[
 \calE_0=\{\Ph^jT_x\mid x\in C_0,\ j\in\bbZ\}.
\]
Every object added to $\{T_x\mid x\in C_0\}$ is therefore
isomorphic to one of the $T_x$. Hence $\calE$ is a small
tilting subcategory, $\Ph(\calE)=\calE$. {Transporting the restriction of $U$
to the full subcategory on the chosen projectives $P_x$ $(x\in C_0)$ along the
chosen isomorphisms $T_x\cong U(P_x)$ gives an equivalence}
\[
 F:C\longrightarrow\calE,\qquad x\longmapsto T_x.
\]
We do not require $\calE$ to be basic.

{The chosen isomorphisms $T_x\cong U(P_x)$ and
$\Ps P_x\cong P_{\ps x}$ transport $\et$ to a natural isomorphism
$\et_F:\Ph F\To F\ps$. Applying the construction above to $\et_F$
makes $F$ a $\bbZ$-equivariant equivalence for the actions generated
by $\ps$ and $\Ph|_\calE$.}

\medskip
\noindent\textbf{\proofstep{step:repetitive-categories} Passage to the repetitive categories.}
{We continue to use the formal additive-hull models, also for
$\prj\widehat B$, with the actions on complexes given termwise.}
\begingroup
For each $i\in\bbZ$, define the additive functor
\[
 L_i:\prj B\longrightarrow\prj\widehat B,
 \qquad B(\blank,x)\longmapsto\widehat B(\blank,x^{[i]}).
\]
{The object assignment is made for every $x\in B_0$. On morphisms between representable projectives,
$L_i$ is induced by $f\mapsto f^{[i]}$ for all $x,y\in B_0$ and $f\in B(x,y)$; extend it to finite direct sums
and direct summands.}
For each complex $X$ of finitely generated projective
$B$-modules {and each $r\in\bbZ$,} put
\[
 (X^{[i]})^r=L_i(X^r),\qquad
 d_{X^{[i]}}^r=L_i(d_X^r).
\]
We call $X\mapsto X^{[i]}$ \emph{placement in the $i$th layer}.
It does not change cohomological degrees and is distinct
from the shift $X[i]$.
The terms same-layer and adjacent-layer also apply to
morphisms $X^{[i]}\to Y^{[i]}$ and
$X^{[i]}\to Y^{[i+1]}$, respectively.
\par\endgroup
Let $\widehat{\calE}$ be the full subcategory of
$\Kb(\prj\widehat B)$ with object set
\[
 \widehat{\calE}_0
 =\{X^{[i]}\mid X\in\calE_0,\ i\in\bbZ\}.
\]

For all $X,Y\in\calE$ and all $i,j\in\bbZ$,
the definition of $\widehat B$ gives natural isomorphisms
of complexes
\[
 \Hom_{\widehat B}\up(X^{[i]},Y^{[j]})
 \cong
 \begin{cases}
  \Hom_B\up(X,Y),&j=i,\\
  D\Hom_B\up(Y,X),&j=i+1,\\
  0,&j\notin\{i,i+1\},
 \end{cases}
\]
where the second line uses the usual differential
 of the dual complex.
Since $\calE$ is tilting, these complexes have zero
 cohomology in every nonzero degree.
Moreover, $\calE$ generates $\Kb(\prj B)$, so
\[
 B(\blank,x)\in\operatorname{thick}(\calE)
 \qquad(x\in B_0).
\]
For each $i\in\bbZ$, the functor $X\mapsto X^{[i]}$
preserves shifts, triangles, and direct summands. Hence
\[
 \begin{aligned}
 \widehat B(\blank,x^{[i]})
 &=\bigl(B(\blank,x)\bigr)^{[i]}\\
 &\in\operatorname{thick}\{X^{[i]}\mid X\in\calE\}\\
 &\subseteq\operatorname{thick}(\widehat{\calE})
 \qquad(x\in B_0).
 \end{aligned}
\]
As $x\in B_0$ and $i\in\bbZ$ vary, the modules
$\widehat B(\blank,x^{[i]})$ form all the indecomposable
projective $\widehat B$-modules, up to isomorphism.
Thus $\widehat{\calE}$ generates $\Kb(\prj\widehat B)$.
Hence $\widehat{\calE}$ is a tilting subcategory
 of $\Kb(\prj\widehat B)$.

{For the same $X,Y\in\calE$ and $i,j\in\bbZ$, taking degree-zero cohomology in the displayed formula gives}
\[
 \Kb(\prj\widehat B)(X^{[i]},Y^{[j]})
 \cong
 \begin{cases}
  \calE(X,Y),&j=i,\\
  D\calE(Y,X),&j=i+1,\\
  0,&j\notin\{i,i+1\}.
 \end{cases}
\]
These identifications also describe composition. For
$\la\in D\calE(Y,X)$, $a:X'\to X$, and $b:Y\to Y'$
in $\calE$, the dual bimodule action is
\[
 (b\la a)(h)=\la(ahb)
 \qquad\bigl(h\in\calE(Y',X')\bigr).
\]
The composite of two adjacent-layer morphisms is zero.
Thus $\widehat{\calE}$ has the morphisms and composition
of the repetition of $\calE$.

For the equivalence $F:C\to\calE$ constructed in {(\ref{step:repetitive-bases})}, {and all $x,y\in C_0$,} put
\[
 F_{x,y}:C(x,y)\xrightarrow{\ \cong\ }\calE(Fx,Fy),
 \qquad f\longmapsto F(f).
\]
Its extension
$\widehat F:\widehat C\to\widehat{\calE}$ is given on objects {for all $x\in C_0$ and $i\in\bbZ$} by
\[
 \widehat F(x^{[i]})=(Fx)^{[i]}.
\]
On same-layer morphisms it uses $F_{x,y}$, whereas on
adjacent-layer morphisms it uses
\[
 DC(y,x)\xrightarrow{\ D(F_{y,x}^{-1})\ }
 D\calE(Fy,Fx),
 \qquad
 \la\longmapsto\la\circ F_{y,x}^{-1}.
\]
The dual bimodule formula above shows that these maps
respect composition. Since $F$ is fully faithful and
essentially surjective, so is $\widehat F$.

We next make the equivariance explicit. Let $G_B$ denote
the autoequivalence of $\Kb(\prj\widehat B)$ induced by $g_B$.
{For all $X\in\calE$, $x\in C_0$, and $i\in\bbZ$, we have}
\[
 G_B(X^{[i]})=(\Ph X)^{[i+n]},
 \qquad
 g_C(x^{[i]})=(\ps x)^{[i+n]}.
\]
On same-layer morphisms, $G_B$ uses $\Ph$; on adjacent-layer
morphisms, it uses the inverse dual of the corresponding
isomorphism on morphism spaces, together with the shift
of the layer index by $n$.
{Let $\et_F:\Ph F\To F\ps$ be the natural isomorphism
constructed in (\ref{step:repetitive-bases}).} Define
\[
 \widehat\et:\ G_B\widehat F\To\widehat Fg_C,
 \qquad
 \widehat\et_{x^{[i]}}
   ={((\et_F)_x)^{[i+n]}}:
   (\Ph Fx)^{[i+n]}\longrightarrow(F\ps x)^{[i+n]}.
\]
Thus its components are obtained by placing the components
of {$\et_F$} in the appropriate layer. Its naturality means that,
for every morphism $u:z\to z'$ in $\widehat C$,
\[
 \widehat F(g_Cu)\,\widehat\et_z
 =
 \widehat\et_{z'}\,G_B(\widehat F(u)).
\]
For same-layer morphisms, this is the naturality of {$\et_F$}.
For adjacent-layer morphisms, it follows from the same
naturality identity by the inverse-dual construction above
and the dual bimodule formula.
The construction is independent of the layer index, so it
commutes with translation of the layers.
{Applying the construction preceding \eqref{eq:Z-equivariance}
to $\widehat\et$ makes $\widehat F$ a $\bbZ$-equivariant equivalence
for the actions generated by $g_C$ and $G_B|_{\widehat{\calE}}$.}

Finally, $\Ph(\calE)=\calE$ and
$G_B(X^{[i]})=(\Ph X)^{[i+n]}$ give
\[
 G_B(\widehat{\calE})=\widehat{\calE}.
\]
The action on $\widehat{\calE}$ is the restriction of $G_B$,
so its inclusion into $\Kb(\prj\widehat B)$ is equivariant.

\medskip
\noindent\textbf{\proofstep{step:repetitive-covering} Application of the covering theorem.}
By {(\ref{step:repetitive-categories})}, $\widehat{\calE}$ is a small $G$-stable
tilting subcategory of $\Kb(\prj\widehat B)$, and
$\widehat F:\widehat C\to\widehat{\calE}$ is
a $G$-equivariant equivalence.
An equivariant quasi-inverse gives the equivalence
$\widehat{\calE}\to\widehat C$ required in
\thmref{thm:Asa11-covering}.
Thus that theorem applies with
\[
 (R,S,G,\calE)
 =(\widehat B,\widehat C,\bbZ,\widehat{\calE})
\]
and gives a derived equivalence between
$\widehat B/\langle g_B\rangle$ and
$\widehat C/\langle g_C\rangle$.

Since $n\geq1$, these actions are free on objects, and
each orbit has a representative in the layers
$0,\ldots,n-1$.
The corresponding finite full subcategories are skeletons
of the orbit categories and have associated algebras
$T_\ph^n(B)$ and $T_\ps^n(C)$.
The tilting equivalence therefore restricts to
\[
 \Db(\mod T_\ph^n(B))
 \simeq
 \Db(\mod T_\ps^n(C)),
\]
as required.
\end{proof}

\section{Some standard facts used in the proofs}
\label{app:standard-facts}

We recall some standard facts used in the main text,
including the reasons for the identifications involved.

\subsection{Picard groups of basic algebras}
\label{app:picard-basic}

For a finite-dimensional $\k$-algebra $B$, let $\Pic_\k(B)$
denote the group of isomorphism classes of invertible
$\k$-central $B$-$B$-bimodules, with multiplication induced
by $\otimes_B$.
Here $\k$-central means that the two actions of $\k$ agree.
{For each $\ph\in\Aut_\k(B)$,} write ${}_1B_\ph$ for the bimodule
whose underlying vector space is $B$ and whose actions are
\[
 b\cdot x\cdot c=bx\ph(c){\qquad\text{for all }b,x,c\in B.}
\]

\begin{prp}\label{prop:basic-picard}
If $B$ is basic, the assignment $\ph\mapsto{}_1B_\ph$
induces an isomorphism
\[
 \Out_\k(B)\xrightarrow{\ \sim\ }\Pic_\k(B).
\]
\end{prp}

\begin{proof}
We use left modules in this proof.
{For any $\ph,\ps\in\Aut_\k(B)$, the bimodule isomorphism}
\[
 {}_1B_\ph\otimes_B{}_1B_\ps
 \longrightarrow{}_1B_{\ph\ps},
 \qquad x\otimes y\longmapsto x\ph(y),
\]
shows that the assignment is a group homomorphism.

We next show where the assumption that $B$ is basic enters.
Let $M$ be an invertible $\k$-central $B$-$B$-bimodule.
The functor $F=M\otimes_B-$ is an equivalence of left
module categories, so it permutes the isomorphism classes
of indecomposable projective modules.
Since $B$ is basic, its left regular module has a decomposition
\[
 {}_BB\cong P_1\oplus\cdots\oplus P_r,
\]
where the $P_i$ represent all these classes, each occurring
exactly once. Consequently,
\[
 {}_BM\cong F({}_BB)
 \cong\bigoplus_{i=1}^rF(P_i)
 \cong\bigoplus_{i=1}^rP_i
 \cong{}_BB.
\]

Choose such a left module isomorphism and use it to transport
the right action of $B$ on $M$ to $B$.
Every endomorphism of the left regular module is right
multiplication by a unique element of $B$.
Thus the transported action has the form
\[
 x\cdot b=x\ph(b){\qquad\text{for all }x,b\in B}
\]
{for some unital $\k$-algebra homomorphism $\ph:B\to B$.}
Moreover, full faithfulness of $F$ gives an isomorphism
\[
 \End_B({}_BB)\xrightarrow{\ \sim\ }\End_B({}_BM),
\]
which sends right multiplication by $b$ on $B$ to the
right action of $b$ on $M$, {for every $b\in B$.}
Under the chosen identification ${}_BM\cong{}_BB$,
this is precisely the map induced by $\ph$.
Hence $\ph$ is an automorphism and $M\cong{}_1B_\ph$.
This proves surjectivity.

{Finally, for any $\ph,\ps\in\Aut_\k(B)$,} a left module isomorphism
${}_1B_\ph\to{}_1B_\ps$ has the form $x\mapsto xu$
for a unit $u\in B^\times$.
It is a bimodule homomorphism if and only if
\[
 \ph(b)u=u\ps(b)\qquad\text{for every }b\in B,
\]
or equivalently $\ps(b)=u^{-1}\ph(b)u$.
Thus the two bimodules are isomorphic exactly when
$\ph$ and $\ps$ have the same image in $\Out_\k(B)$.
This proves the assertion.
\end{proof}

This is the isomorphism underlying the identification
$\Out_\k(B)=\Pic_\k(B)$ in the main text.
For the right-module convention used there, the corresponding
autoequivalence {of $\Db(\mod B)$} is
\[
 X\longmapsto X\otimes_B{}_\ph B_1,
\]
where ${}_\ph B_1$ has left action $b\cdot x=\ph(b)x$
and the usual right action.
It is the inverse bimodule of ${}_1B_\ph$.
With this convention, multiplication of automorphisms
corresponds to composition of autoequivalences.

We also explain the categorical interpretation. {Here $B$ is basic, as in \prpref{prop:basic-picard}.}
Choose a complete set $e_1,\ldots,e_r$ of primitive
orthogonal idempotents of $B$, and let $\calC$ be the category
with objects $1,\ldots,r$ and
\[
 \calC(i,j)=e_jBe_i{\qquad\text{for all }i,j\in\{1,\ldots,r\}.}
\]
A {$\k$-linear category automorphism} of $\calC$ gives an algebra
automorphism of $B$ that permutes the chosen idempotents.

Every outer automorphism class has such a representative.
Indeed, for $\al\in\Aut_\k(B)$, put $f_i=\al(e_i)$ {for each $i\in\{1,\ldots,r\}$.}
Since $B$ is basic, there is a permutation $\si$ {of $\{1,\ldots,r\}$} such that
$Bf_i\cong Be_{\si(i)}$ {for every $i\in\{1,\ldots,r\}$.}
The direct sum of chosen isomorphisms is an automorphism
of the left regular module, hence is right multiplication
by a unit $v$.
{For every $i\in\{1,\ldots,r\}$, it follows that}
\[
 f_iv=ve_{\si(i)},\qquad
 v^{-1}\al(e_i)v=e_{\si(i)}.
\]
Thus composing $\al$ with an inner automorphism gives
an automorphism that permutes the chosen idempotents.

Lastly, let $\ph,\ps$ be {$\k$-linear category automorphisms} of $\calC$.
A natural isomorphism $\et:\ph\To\ps$ consists of
invertible morphisms
\[
 u_i\in e_{\ps(i)}Be_{\ph(i)}{\qquad\text{for each }i\in\{1,\ldots,r\}}
\]
{such that $\ps(a)u_i=u_j\ph(a)$ for all $i,j\in\{1,\ldots,r\}$ and all $a\in e_jBe_i$.}
The element {$u=\sum_{i=1}^r u_i$} is a unit of $B$, and these
equalities say precisely that
\[
 \ps(b)=u\ph(b)u^{-1}\qquad(b\in B).
\]
Conversely, any such inner relation gives a natural
isomorphism, with components $u_i=ue_{\ph(i)}$ {for each $i\in\{1,\ldots,r\}$.}
Consequently, $\Out_\k(B)$, and hence $\Pic_\k(B)$,
can be identified with the group of natural-isomorphism
classes of automorphisms of $\calC$.
No assumption that these automorphisms fix every object
is needed.

\section{Equivalences of stable categories without
 triangle structures}
\label{app:additive-stable}

The main text uses stable equivalence in the sense of
 a $\k$-linear triangle equivalence.
In this appendix we show that the classification in
 \thmref{thm:main} remains valid if one assumes only
 a $\k$-linear equivalence of the stable module categories.
Readers interested only in the equivalence between
 conditions {(\ref{item:main-derived})} and {(\ref{item:main-type})} of that theorem may omit this appendix:
 none of the arguments in the main text depends on it.

The point is that a $\k$-linear equivalence preserves
 the Serre functor, even when it is not assumed to preserve
 the suspension functor.
We use a power of the Serre functor to recover the twist
 on a full transjective component.
The theorem of Miyachi--Yekutieli then provides the
 standard derived equivalence needed to interpret this
 recovery as conjugacy in the derived Picard group.

\subsection{The Serre functor on a full component}

We first record that the component arguments used above
 apply to every affine hereditary type.

\begin{lem}\label{lem:affine-full-return}
Let $H=\k Q$ be connected hereditary of affine type,
 let $\ph\in\Aut(H)$, and let {$n\geq1$ be an integer.}
Put $E=T_\ph^n(H)$, and use the notation
 $P:\Db(\mod H)\to\umod E$ and $\calB$
 from Section~\ref{sec:mesh}.
Then the conclusions of \lemref{lem:full} and
 \lemref{lem:return} hold for $E$.
Moreover, there are natural isomorphisms
\[
\ta_E P\cong P\ta_H,\quad \text{and}\quad
N_E^nP|_\calB\cong P\Ph^-|_\calB.
\]
\end{lem}

\begin{proof}
We explain why the arguments do not require type
 $\widetilde A$.
The covering formula in \eqref{eq:stablecover} holds
 for every hereditary $H$, with
\[
 G=\Ph\ta_H^n[2n].
\]
Every object $Z\in\calB$ is isomorphic to exactly one
 of the following two kinds of objects:
\begin{enumerate}
\item\label{item:component-preprojective} an indecomposable preprojective $H$-module,
      regarded as a complex concentrated in degree zero;
\item\label{item:component-preinjective} $I[-1]$, where $I$ is an indecomposable
      preinjective $H$-module.
\end{enumerate}
Consequently, there is a unique integer
 $c\in\{0,-1\}$ and, up to isomorphism, a unique
 indecomposable $H$-module $N$ such that $Z\cong N[c]$.
If $c=0$, then $N$ is preprojective;
if $c=-1$, then $N$ is preinjective.
The uniqueness of $c$ follows from the degree of the
 unique nonzero cohomology module of $Z$.

Now fix $X,Y\in\calB$ and $j\in\bbZ\setminus\{0\}$.
Both $\Ph$ and $\ta_H$ preserve $\calB$ up to isomorphism.
Applying the preceding description to
 $\Ph^j\ta_H^{nj}X$ and $Y$, we obtain uniquely
 determined integers $a,b\in\{0,-1\}$ and
 indecomposable $H$-modules $M,L$, unique up to
 isomorphism, such that
\[
 \Ph^j\ta_H^{nj}X\cong M[a],
 \qquad Y\cong L[b].
\]
Here $M$ is preprojective if $a=0$ and preinjective
 if $a=-1$; likewise, $L$ is preprojective if $b=0$
 and preinjective if $b=-1$.
Then
\[
{ \Db(\mod H)(G^jX,Y)}
 \cong\operatorname{Ext}^{\,b-a-2nj}_H(M,L).
\]
Since $H$ is hereditary, the right-hand side vanishes
 unless the exponent is $0$ or $1$.
For $j\ne0$, the only possible exception is
 $j=-1$, $n=1$, $a=0$, and $b=-1$.
In that case it is an extension from a preprojective
 module to a preinjective module, and is again zero.
Thus the vanishing argument of \lemref{lem:vanish}
 applies to every affine type.
The covering formula therefore proves full faithfulness
 of $P|_\calB$ exactly as in \lemref{lem:full}.

The remaining component argument also applies unchanged.
The transjective components {in $\Db(\mod H)$} are $\calB[t]$ {for all $t\in\bbZ$,}
 and $G$ sends $\calB[t]$ to $\calB[t+2n]$.
The local support-finiteness and component calculation
 used in the proof of \lemref{lem:full} hold for every
 affine hereditary algebra.
Hence the stable Auslander--Reiten quiver of $E$
 has exactly $2n$ transjective components, and all
 other components are $\ta$-periodic.
The functor $\Om_E$ permutes the transjective
 components transitively.

Finally, the proof of \lemref{lem:return} uses
 repetitive duality and its compatibility with
 $\widehat\ph$ and $\nu_H$, not the underlying graph
 of $Q$. It gives natural isomorphisms
\[
 N_EP\cong P\ta_H[2],
 \qquad
 N_E^nP|_\calB\cong P\Ph^-|_\calB.
\]
Since $P$ is a triangle functor,
 $\Om_EP\cong P[-1]$.
Using the first of the above natural isomorphisms
and $\ta_E\cong\Om_E^2N_E$, we obtain
$$
\ta_EP\cong\Om_E^2 N_EP\cong\Om_E^2P\ta_H[2]\cong P[-2]\ta_H[2]\cong P\ta_H.
$$
\end{proof}

Recall that a {\emph{Serre functor}} on a $\k$-linear category $\calD$ {with
finite-dimensional morphism spaces} is an autoequivalence
 $S$ equipped with bifunctorial isomorphisms
\[
{ D\calD(X,Y)\cong\calD(Y,SX)\qquad\text{for all }X,Y\in\calD.}
\]
This definition uses only the $\k$-linear category,
 not a triangle structure.
Consequently, if $V:\calD\to\calD'$ is a $\k$-linear
 equivalence and $S,S'$ are Serre functors {on $\calD,\calD'$, respectively,} then
\begin{equation}\label{eq:additive-Serre}
 VS\cong S'V
\end{equation}
naturally: transporting $S$ along $V$ gives a Serre
 functor on $\calD'$, and Serre functors are unique
 up to natural isomorphism.

For a finite-dimensional selfinjective algebra $E$,
 the stable module category has Serre functor
\[
 S_E\cong\Om_EN_E.
\]
Indeed, Auslander--Reiten duality and the triangle
 structure of $\umod E$ identify, {for all $X,Y\in\umod E$,}
 $D\uHom_E(X,Y)$ with
 $\uHom_E(Y,\Om_E^{-1}\ta_EX)$, and
 $\Om_E^{-1}\ta_E\cong\Om_EN_E$.
The existence of this triangle structure is used here
 to identify the Serre functor; an equivalence between
 stable categories need not preserve that structure
 in order to satisfy \eqref{eq:additive-Serre}.

\begin{lem}\label{lem:Serre-return}
Under the hypotheses of \lemref{lem:affine-full-return},
 there is a natural isomorphism
\begin{equation}\label{eq:Serre-return}
 S_E^{2n}P|_\calB
 \cong P\ta_H^n\Ph^-|_\calB.
\end{equation}
In particular, $S_E^{2n}$ preserves the full component
 $\calC_\ph$ up to isomorphism.
\end{lem}

\begin{proof}
The Nakayama functor commutes naturally with syzygy.
Since $S_E\cong\Om_EN_E$ and
 $\ta_E\cong\Om_E^2N_E$, we obtain
\[
 S_E^{2n}\cong\ta_E^nN_E^n.
\]
Using \lemref{lem:affine-full-return}, we therefore have
\[
 \begin{aligned}
 S_E^{2n}P|_\calB
 &\cong\ta_E^nN_E^nP|_\calB\\
 &\cong\ta_E^nP\Ph^-|_\calB\\
 &\cong P\ta_H^n\Ph^-|_\calB.
 \end{aligned}
\]
Both $\ta_H$ and $\Ph^-$ preserve $\calB$
 up to isomorphism, proving the last assertion.
\end{proof}

\subsection{Recovery of the type}

\begin{prp}\label{prop:additive-stable-type}
Let $\La$ and $\La'$ be as in \thmref{thm:main}.
If their stable module categories are equivalent as
 $\k$-linear categories, then
 $\typ(\La)=\typ(\La')$.
No compatibility of the given equivalence with
 the suspension functors is assumed.
\end{prp}

\begin{proof}
\setcounter{proofstep}{0}
\noindent\textbf{\proofstep{step:additive-reduction} Reduction to hereditary bases.}
Use \prpref{prop:reduction} in type $\widetilde A$
 and the hereditary reduction recalled in
 Section~\ref{sec:tree} for the tree types.
These reductions are derived equivalences and hence
 induce triangle equivalences of stable categories.
Composing with them preserves the hypothesis of a
 $\k$-linear equivalence of stable categories.
It therefore suffices to treat hereditary extensions
\[
 E=T_\ph^n(H),\qquad E'=T_\ps^{n'}(H').
\]
{Here $H,H'$ are connected hereditary algebras of affine type, $\ph\in\Aut(H)$, $\ps\in\Aut(H')$, and
$n,n'$ are positive integers.}
Let $V:\umod E\to\umod E'$ be a $\k$-linear equivalence.

\medskip
\noindent\textbf{\proofstep{step:additive-type} Recovery of the hereditary type and $n$.}
An equivalence of stable module categories of selfinjective
 algebras with no block of Loewy length two induces an
 isomorphism of their stable Auslander--Reiten translation
 quivers; see \cite[Corollary~X.1.9]{ARS}.
The algebras $E$ and $E'$ satisfy this hypothesis.
Indeed, they are connected and representation-infinite,
 whereas a connected selfinjective algebra of Loewy
 length two is a representation-finite Nakayama algebra.

Thus $V$ preserves the translation quiver, including
 arrow multiplicities and the translation on vertices.
By \lemref{lem:affine-full-return}, the number of
 nonperiodic components is $2n$ for $E$ and $2n'$ for $E'$.
Hence $n=n'$.
The reconstruction from one such translation quiver
 in the proof of \prpref{prop:recover} also applies,
 and gives the same hereditary type $\De$.
We may therefore use the common chosen hereditary
 model $H=H_\De$ for both extensions.

This use of Auslander--Reiten translation concerns
 vertices of the translation quiver only.
We do not infer a natural isomorphism
 $V\ta_E\cong\ta_{E'}V$ from it.

\medskip
\noindent\textbf{\proofstep{step:additive-component} Comparison on a full component.}
Write
\[
 P_\ph:\calB\longrightarrow\calC_\ph,
 \qquad
 P_\ps:\calB\longrightarrow\calC_\ps
\]
for the full-component equivalences supplied by
 \lemref{lem:affine-full-return}.
Since $\Om_{E'}$ acts transitively on the transjective
 components, we may compose $V$ with a suitable power
 of $\Om_{E'}$ so that it sends $\calC_\ph$ to
 $\calC_\ps$ up to isomorphism.
This does not require $V$ to commute with syzygy.

Choose a quasi-inverse of $P_\ps$ and obtain a
 $\k$-linear equivalence $U:\calB\to\calB$ with
\[
 P_\ps U\cong VP_\ph.
\]
The equivalence $V$ preserves the Serre functors by
 \eqref{eq:additive-Serre}.
Applying \eqref{eq:Serre-return} to both sides gives
\[
 \begin{aligned}
 P_\ps U\ta_H^n\Ph^-
 &\cong VP_\ph\ta_H^n\Ph^-\\
 &\cong VS_E^{2n}P_\ph\\
 &\cong S_{E'}^{2n}VP_\ph\\
 &\cong S_{E'}^{2n}P_\ps U\\
 &\cong P_\ps\ta_H^n\Ps^-U.
 \end{aligned}
\]
Here all functors are restricted to $\calB$ where
 appropriate.
Since $P_\ps$ is fully faithful, we obtain a natural
 isomorphism
\begin{equation}\label{eq:additive-component-comparison}
 U\ta_H^n\Ph^-|_\calB
 \cong\ta_H^n\Ps^-U.
\end{equation}

\medskip
\noindent\textbf{\proofstep{step:additive-DPic} Passage to the derived Picard group.}
At this point $U$ is only an equivalence of the full
 category $\calB$.
We use \cite[Theorem~3.7]{MY}, in the form stated in
 \prpref{prop:MY}, to obtain a standard derived
 autoequivalence $\widetilde U$ of $H$ which preserves
 the chosen transjective component and whose restriction
 to $\calB$ is naturally isomorphic to $U$.

The same theorem says that restriction is injective
 on the subgroup of $\DPic_\k(H)$ preserving this
 component.
Indeed, this subgroup corresponds to the first factor
 in \eqref{eq:mesh}; the second factor records shifts
 to other transjective components.
Both sides of \eqref{eq:additive-component-comparison}
 are restrictions of standard derived autoequivalences
 preserving the chosen component.
Consequently,
\[
 \widetilde U\ta_H^n\Ph^-
 \cong\ta_H^n\Ps^-\widetilde U
\]
as derived functors.

The class of $\ta_H$ is central in $\DPic_\k(H)$
 by \cite[Lemma~3.1]{MY}.
We can therefore cancel $\ta_H^n$ and obtain
\[
 \widetilde U\Ph^-\cong\Ps^-\widetilde U,
 \qquad\text{equivalently}\qquad
 \widetilde U\Ph\widetilde U^-\cong\Ps.
\]
Thus the twists are conjugate in $\DPic_\k(H)$.
Together with {(\ref{step:additive-type})}, this proves equality of all three
 components of the type.

For generalized multifold extensions with nonzero
 integer jump, first apply
 \thmref{thm:normalization} as in the main text.
The preceding argument then applies to the resulting
 twisted multifold extensions.
\end{proof}

\begin{rmk}\label{rmk:additive-stable-classification}
Under the hypotheses of \thmref{thm:main}, consider
 the additional condition
\[
 \textup{{(\ref{item:main-stable}')}}\qquad
 \umod\La\simeq\umod\La'
 \quad\text{as $\k$-linear categories}.
\]
By the main theorem and \prpref{prop:additive-stable-type},
 we have
\[
 \textup{{(\ref{item:main-derived})}}\ \Longrightarrow\ \textup{{(\ref{item:main-stable})}}
 \ \Longrightarrow\ \textup{{(\ref{item:main-stable}')}}
 \ \Longrightarrow\ \textup{{(\ref{item:main-type})}}
 \ \Longrightarrow\ \textup{{(\ref{item:main-derived})}}.
\]
Hence these four conditions are equivalent.
In particular, the classification remains valid when
 condition {(\ref{item:main-stable})} is replaced by {(\ref{item:main-stable}')}, without assuming
 that the stable-category equivalence is a triangle functor.
The same observation applies to the family in
 {\corref{cor:HH-family}.}

\end{rmk}

\section*{Acknowledgments}
The author has been working on the questions considered here since 2015.
This work was partially supported by JSPS Grant-in-Aid for
Challenging Exploratory Research 25610003, Scientific Research (B)
25287001, Scientific Research (C) 18K03207 and 25K06922, and by
Osaka Central Advanced Mathematical Institute (MEXT Promotion of
Distinctive Joint Research Center Program JPMXP0723833165).

\section*{Declaration of AI assistance}
The author used OpenAI's ChatGPT to assist with mathematical
discussions, the formulation and examination of proof arguments,
literature searches, and the preparation and revision of the text
and LaTeX source.
The author takes full responsibility for the mathematical claims,
proofs, references, and all other content of this paper.

\end{document}